\documentclass[11pt]{article}
\usepackage{amsfonts,amssymb,amsbsy,amsmath,amsthm,eufrak,dsfont,bbold}
\usepackage{graphicx}
\usepackage{color}

\numberwithin{equation}{section}

\usepackage[makeroom]{cancel}

\usepackage[english]{babel}
\newtheorem{theorem}{Theorem}[section]
\newtheorem{proposition}[theorem]{Proposition}
\newtheorem{lemma}[theorem]{Lemma}
\newtheorem{follow}[theorem]{Corollary}

\newtheorem{pr}[theorem]{Proposition}
\theoremstyle{definition}
\newtheorem{remark}[theorem]{Remark}

\newcommand{\bel}{\begin{equation} \label}
\newcommand{\ee}{\end{equation}}
\newcommand{\ba}{\begin{array}}
\newcommand{\ea}{\end{array}}

\newcommand{\Z}{{\mathbb Z}}
\newcommand{\R}{{\mathbb R}}
\newcommand{\N}{{\mathbb N}}
\newcommand{\C}{{\mathbb C}}
\newcommand{\re}{{\mathbb R}}

\begin{document}
\begin{center}{\Large \bf Threshold Singularities of the Spectral Shift Function for a Half-Plane Magnetic  Hamiltonian}

\medskip

{\sc Vincent Bruneau\footnote{ 
Institut de Math\'ematiques de Bordeaux,
UMR 5251 du CNRS,
Universit\'e de Bordeaux,
351 cours de la Lib\'eration,
33405 Talence Cedex, France. 
e-mail: vbruneau@math.u-bordeaux.fr}  and  Pablo Miranda\footnote{
Departamento de Matem\'atica y Ciencia de la Computaci\'on, Universidad de Santiago de Chile, Las Sophoras 173. Santiago, Chile. e-mail: pablo.miranda.r@usach.cl}}

\medskip
\medskip\medskip
\end{center}

\abstract We consider the Schr\"odinger operator  {with} constant magnetic field defined  on the  half-plane with {a} Dirichlet boundary condition, $H_0$, and a decaying electric perturbation $V$. We analyze the spectral density near {the Landau levels}, {which  are} thresholds 
in the spectrum {of $H_0,$}  by studying the  Spectral Shift Function (SSF) associated  to the pair $(H_0+V,{H_0})$. For perturbations  of {a} fixed sign{,}  we {estimate}
the SSF in terms of the eigenvalue counting function {for}  certain compact operators. 
If  {the decay of} $V$ is   power-like, then using {p}seudodifferential analysis, we {deduce} that   there are  singularities at the thresholds and we obtain the corresponding asymptotic behavior {of the SSF}.   
Our {technique gives} also results for the Neumann  boundary condition.

$ $


\medskip



\section{Introduction}\setcounter{equation}{0}
\label{s1}

\subsection*{Motivations}
We consider $H_0^D$ (resp., $H_0^N$),  the Dirichlet (resp.,  Neumann) self-adjoint realization of the magnetic Schr\"odinger operator
\bel{Hk}
    -\frac{\partial^2}{\partial x^2} +
    \left(-i\frac{\partial}{\partial y} - bx\right)^2, \quad b>0,
    \ee
    in the half-plane $\R_+\times\R$ {($\R_+:=(0,\infty)$)}.

Our goal is to analyze the effect{s on the spectrum when a relatively compact perturbation of $H_0^D$ or  $H_0^N$ is introduced.   The perturbations under consideration  will be}   real electric potentials $V$ {that} {decays {to zero} at infinity} in $\R_+\times\R$.

Such {effect} is now well understood for perturbations of the so-called {\it Landau Hamiltonian} $H_L$ i.e.{,} the  magnetic Schr\"odinger operator \eqref{Hk} {but defined for  the whole plane} $\R^2$.  
The {Landau Hamiltonian} admits pure point spectrum with eigenvalues of infinite multiplicity  (the so called {\it Landau levels} ${\mathcal E}_j$, $j \in \N$
). It is established that perturbations by a {decaying}  electric potential of {a} definite sign, even if it is compactly supported,  produce {an accumulation} of { discrete} eigenvalues {around} the Landau levels (see  \cite{rai1, ivrii2, raiwar, melroz, filpush, rozta2}).  {Using variational methods, it can be seen that} the distribution of these eigenvalues is governed by the counting function of the eigenvalues of the {compact} Toeplitz operators $P_jVP_j$, where $P_j$ is the spectral projection on Ker$(H_L- {\mathcal E}_j)$. Then, depending on the {decay} rate  of $V${,}  it is possible to obtain {the} asymptotic behavior of the counting functions of {the} eigenvalues {of $H_L+V,$} near each Landau Level.  
{Tools from} pseudodifferential {analysis}, {together with} variational and Tauberian  {methods}  have been used to obtain this behavior for V: power-like decaying \cite{rai1, ivrii2},  exponentially decaying \cite{raiwar}, compactly supported \cite{raiwar, melroz, filpush, pushroz2, roz}. Also 
magnetic \cite{rozta2} and {geometric}  perturbations \cite{pushroz1, per, GoKaPe13} have been considered.


In our case, on the half-plane, the spectrum of $H_0^D$ (resp. $H_0^N$) is {rather} different {from that} of $H_L$. It is {purely absolutely} continuous, given by $\sigma (H_0^D)=[b , \infty)$ (resp. $\sigma (H_0^N)=[\Theta_0, \infty)$, $0< \Theta_0<b$). From the dynamical point of view, {this difference is related to the fact that} in $\R^2$ the classical trajectories are circles,  while in $\R_+\times \R$ there {exist}  propagation phenomena along the boundary $\{0\}\times \R$. The {accumulation of the} discrete spectrum 
of $H_0^D+V$ and of $H_0^N+V{,}$ for compactly supported potentials $V{,}$ {was} studied in \cite{bmr2}{.}  However,  to our best knowledge, there {are} no results concerning the continuous spectrum.

A natural tool to {extend {this}  notion of} spectral densit{y}, from the discrete spectrum into the continuous spectrum, is the {s}pectral {s}hift {f}unction (SSF)  {(see \eqref{jun_24_16} and \eqref{15_sep_16} below).}
For example, {this function is studied} for the  Schr\"odinger operator with constant  magnetic field in $\R^3$, which {has} purely absolutely continuous  {spectrum}, and  it is proved that it admits  singularities at the Landau levels (see \cite{ferrai, BoBrRa14_02} and   other {magnetic} examples in  \cite{bsoccrai, rai3, tie}).  {The operators considered in these works  admit an analytically fibered decomposition, and the singularities of the SSF are present at the minima   of the corresponding band functions. Is is important  to notice that in all these models, the minima of the band functions are non-degenerate. 
These points are called thresholds, because they are points where  the {Lebesgue} multiplicity of the {a.c.} spectrum  changes. }

From a general point of view, it is known that  the extrema of the band functions play a significant role {in the description of the  spectral properties of fibered  operators}(see \cite{GeNi}). In the particular case where these extrema  are reached and   non{-}degenerate,   there is a well known procedure to obtain effective Hamiltonians that  allows to describe the distribution of eigenvalues (as in \cite{rai2, bmr2})    and the singularities of the SSF (see  \cite{bsoccrai}). 

{Our operator $H_0$ also admits an analytically fibered decomposition (see  \eqref{sant1} below), and as we will see, the singularities of the spectral shift function are  present at the infima of the band functions as well. However, the  nature of these infima  is completely different from the ones mentioned above. They are not reached, they are  the limits of the band functions at infinity. The derivative is zero only at infinity.} This is the source of one of the main technical difficulties  that we have to {overcome} in order to  describe the {properties} of the SSF.




{Since our extremal points are only the limits on the band functions, it is necessary to modify the analysis of previous works.} The  phenomenon of {thresholds given by limits of the band functions} is also present  for some quantum Hall effect models (see \cite{bmr1}) and for  some Iwatsuka models (see \cite{mir}). 
 In  these works, {the SSF}  was studied only  {in the region  where} it  counts the number of  discrete eigenvalues. Similar results are expected {to hold also inside the continuous spectrum}, but it requires a more precise description of  some Birman-Schwinger operators{,} {e}specially when the considered energy levels cross the {corresponding} band functions. 

\subsection*{Dirichlet magnetic Schr\"odinger operator on the half-plane}

In the paper, we will focus on the Dirichlet realization of magnetic Schr\"odinger operator \eqref{Hk} and write $H_0$ instead of  $H^D_0$. {We will discuss the} Neumann boundary condition at {the end} of  {Section} \ref{s2}. 

 The operator $H_0$ is  generated in $L^2(\R_+\times\R)$  by the closure of the quadratic form
$$\int_{\R_+\times\R}\left(\left|\frac{\partial u}{\partial x}\right|^2+\,\left|i\frac{\partial u}{\partial y}+bxu\right|\right)^2dy\,dx, {\quad b>0, }$$
defined originally  on $C_0^\infty (\R_+\times\R)$. 
This is the Hamiltonian of a 2D spinless nonrelativistic quantum {particle}  in a half-plane, subject to a constant  magnetic field of scalar intensity $b$.

Let ${\mathcal F}$ {be} the partial Fourier transform with respect to
$y$, i.e.
$$
({\mathcal F}u)(x,k) = (2\pi)^{-1/2} \int_{\re} e^{-iyk} u(x,y)\,{dy}, \quad (x,k) \in \re_+ \times \re.
$$
Then, we have the identity
\bel{sant1}
{\mathcal F} H_0 {\mathcal F}^* = \int_\R^{\oplus} h(k) \,dk,
\ee
where the operator $h(k)$ is the Dirichlet realization in $L^2(\R_+)$ of
$$
 - \frac{d^2}{dx^2} + (bx-k)^2, \quad k \in \R.
$$
{The domain of the operator $h(k)$ is:
\bel{24apr17}
D(h(k))=\{f\in {\rm H}^2(\R_+);x^2f {\in} L^2(\R_+)\}\cap {\rm H}_0^1(\R_+).
\ee}{Note that} it does not depends on $k$.
{Also}, the family    ${\{}h(k){\}}_{k\in\R}$, is {real} analytic in the sense of Kato  \cite{kato}{, and} for each $k \in \R$ the operator $h(k)$  has a discrete and simple spectrum. Let
$\left\{E_j(k)\right\}_{j=1}^{\infty}$ be the increasing sequence of
the eigenvalues of $h(k)$. For $j\in \N$, the function $E_j(\cdot)$  is called the $j$-th \emph{band function}. By {the} Kato
analytic perturbation theory, $E_j(k)$  is a real analytic functions of $k \in \R$. Further, it is proved in  \cite{bp} (see also {\cite[Chapter 15.A]{ivrii2}}), that for  any $j\in \N$, the band function $E_j$ is strictly decreasing, and
\bel{oct9}
\lim_{k\to -\infty}E_j(k)=\infty, \quad \lim_{k\to \infty}E_j(k)=b(2j-1)=\mathcal{E}_j.
\ee
In consequence, the spectrum of $H_0$ 
is purely absolutely continuous and is given by $$
\sigma(H_0)=\bigcup_{j\in \N}[\mathcal{E}_j,\infty) = [b,\infty).
$$
Using decomposition \eqref{sant1} and the monotonicity of $E_j$, it is possible to see that the {Lebesgue} multiplicity of the {a.c.} spectrum of $H_0$ changes at any point in the set $\{{\mathcal E}_{j}\}_{j=1}^\infty$. {Such  points are called} \emph{thresholds} in  $\sigma(H_0)$.


\subsection*{Perturbation and Spectral Shift Function}

Suppose that the electric potential $V : \re_+ \times \re \to \re$ is a Lebesgue measurable {function} {that} satisfies 
    \bel{cd2}
    |V(x,y)|\leq C \,{\langle x,y \rangle  ^{-m}} 
, \quad (x,y) \in \R_+\times \R,
\ee  for  some  positive constant $C$, $m > 2$, and $\langle x, y \rangle:=(1+x^2+y^2)^{1/2}$. On the domain of $H_0$ introduce the operator 
$$H: = H_0 + V,$$ self-adjoint in $L^2(\re_+ \times \re)$.
Estimate \eqref{cd2} combined  {with} the diamagnetic inequality in the half-plane {(see \cite{hlmw})} imply that for any real $E_0 < \inf \sigma(H)$ the operator
 $|V|^{1/2} (H_0 - E_0)^{-1}$ is Hilbert--Schmidt, and hence the resolvent difference $(H-E_0)^{-1}-(H_0-E_0)^{-1}$ is
a trace-class operator. In particular,  the {absolutely continuous} spectrum of $H$ coincides with $[{\mathcal E}_1,\infty)$. Furthermore, 
there exists a unique function $\xi = \xi(\cdot; H, H_0) \in L^1(\re; (1+E^2)^{-1}dE),$ called the
spectral shift function  for the operator pair $(H,H_0)$, that satisfies the Lifshits-Kre\u{\i}n trace formula
$$\mbox{Tr}(f(H)-f(H_0))=\int_\R \xi(E;H,H_0) f'(E)dE, $$
for each $f \in C_0^\infty(\R)$, and vanishes identically in $(-\infty, \inf \sigma(H))$ \cite{yaf}.

In scattering theory, the SSF can be seen as  the scattering phase of the operator pair $(H,H_0)$, namely  we have the  Birman-Kre\u{\i}n formula
$${\rm det}(S(E))=e^{-2\pi i\xi(E)}, \quad E \in [b,\infty) \, {\rm a.e.}, $$ 
where $S(E)$ is the scattering matrix of the operator pair $(H,H_0)$. In addition, for almost every $E<b=\mathcal{E}_1$ the SSF coincides with the eigenvalue counting function of the operator $H$, i.e. 
\bel{jun_24_16}\xi(E;H,H_0)=-\,{\rm Tr} \,{\bf 1}_{(-\infty,E)}(H),\ee
where ${\bf 1}_\Omega$ denotes the characteristic function of the Borel set $\Omega\subset\R$. 

\vspace*{.5cm}

In this article, for $V$ of {a} definite sign, we investigate  the properties of the  SSF  $\xi(E;H,H_0)$, {particularly its behavior} near the Landau levels. 
Using the Pushnitski representation formula of the SSF, we will reduce our analysis to the study of counting function of some parametrized compact operators (see Theorem \ref{firsteh}).
To prove these  results, in Section \ref{SecThm1} we will describe precisely {some properties of} a Birman-Schwinger operator for which a careful analysis of the band functions and of {their} derivatives {is} necessary.  As a consequence, in Corollary \ref{coro1}, we {can} show that the SSF is bounded on compact sets not containing the thresholds $\{{\mathcal E}_j\}$. Then, in Section \ref{SecThm2}, assuming {that} $V$ {is} smooth and {admits a} power-like {decay} at infinity, we will prove  our main result Theorem \ref{V_power}. It gives the asymptotic behavior of $\xi(E;H,H_0)$ as the energy $E$ approaches the singularit{y} present at the  spectral threshold $\mathcal{E}_j$. {This result is} proved using {pseudodifferential} {methods}. 

%

\section{Main results} \label{s2}
%
%
%

\subsection{Reduction {to a counting function for a compact operator}} \label{ss23}

Our first main result concerns an effective Hamiltonian that permits us to estimate the behavior of the SSF for a wide class of non-negative potentials $V$. The effective Hamiltonian is given by the real part {of a limit} of the trace class operator in \eqref{sof1} below. To describe this operator  we need to introduce some notations, which will be used {systematically in the sequel.}

Fix $k \in\re$ and $j \in {\mathbb N}$. Denote by $\pi_j(k)$ the one-dimensional
orthogonal projection onto ${\rm Ker}\,(h(k) - E_j(k))$. Then
    \bel{lau5}
    \pi_{j}(k)  =\big| \psi_{j}(\cdot;k) \big\rangle \big\langle
\psi_{j}(\cdot;k)\big|,    \ee
    where
 $\psi_j(x;k)$, $x \in \re_+$, is an eigenfunction
of  $h(k)$ that  satisfies
    \bel{defpsij}
h(k) \psi_j(\cdot;k) = E_j(k) \psi_j(\cdot;k), \quad \psi_j(0;k) = 0, \quad
\|\psi_j(\cdot;k)\|_{L^2(\re_+)} = 1.
    \ee
Moreover,  $\psi_j(\cdot;k)$ could be chosen to be real-valued, and {since the family $\{h(k)\}_{k\in\R}$ is analytic}, it can be chosen analytic as a function of $k$ . {The projection $ \pi_{j}$ is analytic with respect to $k$ as well.}

Fix $j \in {\mathbb N}$. For $z \in {\mathbb C}_+ : = \{\zeta \in {\mathbb C} \, | \, {\rm Im}\,\zeta > 0\}$ define
    \bel{sof1}
T_j(z):= |V|^{1/2}{\mathcal F}^* \int^\oplus_\R(E_j(k)-z)^{-1}\pi_j(k)\,dk\,{\mathcal F}|V|^{1/2}.
    \ee
    By Proposition \ref{fom} below, if $V$ satisfies \eqref{cd2}, the limit $\lim_{\delta \downarrow 0}T_j(E+i\delta):=T_j(E)$ exists in the {trace class}-norm for energies $E \in \re$, $E \neq {\mathcal E}_j$.

{For a compact self-adjoint operator $A$, {define} the eigenvalue counting function}
\bel{defn}
n_{\pm}(s; A) : = {\rm Tr}\,{\bf 1}_{(s,\infty)}(\pm A), \quad s>0.
\ee
\begin{theorem}\label{firsteh}
\noindent Assume that $V\geq 0$ satisfies \eqref{cd2}, and write  $H_\pm=H_0\pm V$. For all $j \in  \N$ and  all $r \in (0,1)$ we have
 \bel{pr41}\ba{ll}
    &n_{\mp}(1+r;{\rm Re}\,T_j(\mathcal{E}_j+\lambda))+O(1)  
\\[.5em] \leq &
\pm \xi(\mathcal{E}_j+\lambda;H_\pm,H_0) \\[.5em]
\leq &
n_{\mp}(1-r; {\rm Re}\,T_j(\mathcal{E}_j+\lambda))+O(1),
    \ea\ee
    as $\lambda \to 0$. \end{theorem}
\noindent The proof of Theorem \ref{firsteh} is {contained} in Section \ref{s4}.\\

{Arguing as in  the proof Theorem \ref{firsteh} we have the following: }
\begin{follow}\label{cor1}
Suppose that  $V\geq0$ satisfies \eqref{cd2}. Then, on any compact set $\mathcal{C}\subset \R\setminus\{{\mathcal E}_l\}_{l=1}^\infty$, 
$$\sup_{E\in{\mathcal C}}\xi(E;H_\pm,H_0)<\infty,$$
i.e.  the SSF is bounded  away from the thresholds.

\end{follow}

\subsection{Spectral asymptotics}

One consequence of Corollary \ref{cor1} is that the only possible {points of unbounded growth}  of  $ \xi(E;H,H_0)$ are  the Landau Levels ${\mathcal E}_j$. {On the other side, Theorem \ref{firsteh} can be used to describe the} explicit asymptotic behavior of the SSF at these points. In our second main theorem we obtain this behavior for potentials $V$ that decay moderately at infinity. {To  measure  this  decaying rate it is typically considered  the following volume function:} 
\bel{apr20}
N(\lambda,a):=\frac{1}{2\pi}vol\{(x,\xi)\in \R_+ \times \R; a(x,\xi)> \lambda\},
\ee
{where $a:\R^2\to\R$  is measurable,  $\lambda >0$}, and $vol$ denotes  the Lebesgue measure in $\R^2$. Further, we will need   the H\"ormander class:
\noindent {for} $p, q\in \R$
\bel{12mar17a}\mathcal{S}_p^q:=\left\{ a\in C^\infty(\R^2): \sup_{(x,\xi) \in \R^2}|\langle x\rangle^{p}\langle x,\xi\rangle ^{q+\alpha}\partial_\xi^\alpha \partial_x^\beta a(x,\xi)|<\infty, \, {\textrm{for any}\, \alpha,\beta\in\Z_+}\right\}.\ee
Then,  if $a\in\mathcal{S}_0^m$,  we have that $N(\lambda,a) = O(\lambda^{-\frac{2}{m}})$, for $\lambda \downarrow 0$.   {But this {bound} is insufficient for our purposes, since  it will be   necessary to have also some  control ``from below" and  on the regularity  of  the volume function.  In consequence, we will suppose that $V$ satisfies the following conditions:}

{$\ba{l}\textrm{There {exists}}\,\, m>2\,\, \textrm{such that:} \ea$
\bel{12apr17}\left\{\ba{llr}
{\rm a)} \,\,V \,\,\textrm{is the restriction on} \,\, \R_+\times\R \textrm{ of a function in}\,\, \mathcal{S}_0^m\\[.5em]
{\rm b)} \,\,N(\lambda,V)\geq C \lambda^{-2/m},\,\,\textrm{for some}\,\, C>0\,\, \textrm{and}\,\, 0<\lambda<\lambda_0\\[.5em]
 {\rm c)}\,\,  \displaystyle{\lim_{\epsilon \downarrow 0}\limsup_{\lambda\downarrow 0}\,\lambda^{2/m}\left(N(\lambda(1-\epsilon),V)-N(\lambda(1+\epsilon),V)\right)=0}.
\ea\right.\ee}

Conditions \eqref{12apr17} are commonly assumed in the study of the distribution of eigenvalues of some pseudodifferential operators (see  {for instance} \cite{dauro, rai1, ivrii2, IwaTam98, rozta2}). A  typical situation  of  $V$ satisfying \eqref{12apr17}, is when 
$\lim_{(x,y)\to \infty} \langle x, y\rangle^{m}V(x,y)=\omega\left(\frac{(x,y)}{|(x,y)|}\right),$ where  $\omega: S^1\to [\epsilon, \infty)$ is smooth and  $\epsilon>0$.



In the following, for two functions $F$ and $G$ defined  on some interval $I$, we will write \bel{12apr17b} F(x) \asymp G(x)\quad  \textrm{if}\quad   c_- G(x) \leq F(x) \leq c_+ G(x),\ee for all $x$ in $I$ and for  positive constants $c_\pm$.

\begin{theorem}\label{V_power}  Let  $V\geq0$  and write $H_\pm=H_0\pm V$. If $V$ satisfies \eqref{12apr17}, {then}  the following asymptotic formulas for the SSF
\bel{the4}
\xi({\mathcal E}_j\pm\lambda;H_\pm,H_0)=\pm b\,{N(\lambda,V)}(1+o(1)), \quad \lambda \downarrow 0,
\ee
hold true for any fixed $j\in\N$.  This implies in particular that 
$$
\xi({\mathcal E}_j\pm\lambda;H_\pm,H_0)\asymp \lambda^{-2/m},\quad \lambda \downarrow 0.
$$

\end{theorem}

The proof of Theorem \ref{V_power} can be found in Section \ref{proof_V_power}.

\textbf{Remarks}
\begin{enumerate}
\item The results  in  Theorem \ref{V_power} resemble  those for  the eigenvalue counting function  of $H_{L,V}:=H_L+V$
. More explicitly, if $V\geq 0$ {satisfies} \eqref{cd2} we can define the function that counts the number of eigenvalues in the gaps of 
$\sigma_{\rm{ess}}(H_{L,V})$ as ${\mathcal N_j}(\lambda)=\rm{Tr}({\bf 1}_{({\mathcal E}_j+\lambda,{\mathcal E}_{j+1})}(H_{L,V})). $
In this case, from {Pushnitski formula} \eqref{pssf} and the Birman-Schwinger principle one easily  obtain that 
\bel{15_sep_16}    {\mathcal N_j}(\lambda) =\xi({\mathcal E}_j+\lambda;H_{L,V},H_L)+O(1),\quad{\lambda\downarrow0},\ee
and therefore the study of the eigenvalue counting function of the perturbed Landau Hamiltonian  $H_{L,V}$, is the same as the study of the Spectral Shift Function {for} the pair $(H_{L,V},H_L$). {In fact},  under  conditions similar to those in  Theorem \ref{V_power},  {it is proved in \cite{rai1}} that  the  asymptotic behavior of ${\mathcal N}_j(\lambda)$  is also {given} by a  semiclassical formula of the form \eqref{the4}. {Since $H_L$ can be decomposed as in \eqref{sant1} but with constant band functions, from \eqref{15_sep_16}}   one can say   that, to some extent, our work extends the 2D results on the SSF of \cite{rai1},   to a case where the band functions of the unperturbed operator are not constant.
\item Related to the previous remark, we can mention  the study of  the SSF under   compactly supported perturbations of $H_0$ (including obstacle perturbations), as a natural and important open problem (in particular from the physical point of view, see \cite{bp,HoSmi02,AfHe09,fh}). These cases present a different difficulty  since   the  pseudodifferential analysis {we} used here, does not work (there is no {convenient} class of symbols for compactly supported potentials){,} and some non{-}semiclassical asymptotics are expected (see \cite{raiwar, bmr2}). In fact, using the effective Hamiltonian of our Theorem \ref{firsteh} together with ideas of { \cite{bmr1, bmr2}}, one can show that  for $V\geq 0$ of compact support 
$$
\xi({\mathcal E}_j-\lambda;H_-,H_0)\asymp |\ln \lambda|^{1/2}, \quad \lambda \downarrow 0. 
$$
\item The singularities of the SSF are naturally related to clusters of resonances (see for instance \cite{BoBrRa14_02}). It would be interesting to analyze this phenomenon, but the first difficulty to overcome  is to define the resonances for our fiber operator. Due to the exponential {decay} properties of the band functions, some non-standard tools of complex analysis  would be necessary.\end{enumerate}

{Let us complete the results of Theorem \ref{V_power} by other consequences of our analysis.}

\begin{follow}\label{coro1}\leavevmode
\begin{enumerate} \vspace{.1pt}

\item[{\rm 1.}] Assume that $V\geq0$ satisfies \eqref{cd2}.  Then
$$
\ \xi(\mathcal{E}_j-\lambda;H_+,H_0) =O(1),  \quad \lambda \downarrow 0.
$$

\item[{\rm 2.}]  On the other side, if  $V\geq0$ satisfies \eqref{12apr17} 

\bel{the4_-+}
\xi({\mathcal E}_j+ \lambda;H_-,H_0)=o (\lambda^{-2/m} ) , \quad \lambda \downarrow 0.
\ee

\end{enumerate}
\end{follow}

The proof of part 1 of this Corollary is a direct    consequence of Theorem \ref{firsteh}, while part 2 can be proved using the same tools of pseudodifferential analysis used for the proof of Theorem \ref{V_power}. 

{

{From} Theorem \ref{V_power} and  Corollary \ref{coro1} we {can compare the behavior of the SSF as a Landau level is approached from different sides: 
{\bel{29jun16}
\lim_{\lambda \downarrow 0}\frac{\xi({\mathcal E}_j + \lambda;H_-,H_0)}{\xi({\mathcal E}_j- \lambda;H_-,H_0)}=\lim_{\lambda \downarrow 0}\frac{\xi({\mathcal E}_j - \lambda;H_+,H_0)}{\xi({\mathcal E}_j+ \lambda;H_+,H_0)}=0.
\ee}}
This result is different to  the results obtained for $3D$ magnetic Hamiltonians for which the behavior of the SSF was studied  at the thresholds (see \cite{ferrai}, \cite{rai3}, \cite{tie}). For those models  the corresponding limit  \eqref{29jun16} is a constant different from {zero}, at least for $H_-$}, which gives  a generalization of the Levinson's formula. 

At last, let us mention that following the proofs of our work, it is easy to obtain some analog results for the half-plane magnetic Schr\"odinger operator with a Neumann
 boundary 
condition.

\subsection{Results for Neumann boundary conditions }

Let us consider $H_0^N$, the
Neumann realization of \eqref{Hk}, namely the self-adjoint operator generated in $L^2(\R_+\times\R)$  by the closure of the quadratic form
$$\int_{\R_+\times\R}\left(\left|\frac{\partial u}{\partial x}\right|^2+\,\left|i\frac{\partial u}{\partial y}+bxu\right|\right)^2dy\,dx,$$
defined originally  on ${C_0^\infty} ([0, \infty)\times\R)$, see \cite{fh}.  

A fibered decomposition of the form \eqref{sant1}  holds true in this case as well. Thanks to \cite{dahe} we know that each band function of the Neumann Hamiltonian, $E_j^N$, is a  decreasing function until a unique non degenerated minimum, and then increase {satisfying $\lim_{k\to-\infty} E_j^N(k)=\infty$ and $\lim_{k\to\infty} E_j^N(k)=\mathcal{E}_j$}.  Then,  the minimum  of each  band  is a threshold of the spectrum of $H_0^N$. Due to the  non-degeneracy condition, it is well known how to study the behavior of the SSF at this points (see \cite{bsoccrai}, \cite{bmr2}).

On the other side, for the extremal points at infinity, it can be shown that the behavior of the band functions and of the associated eigenfunctions {are the same as those of the}  Dirichlet operator (see Propositions \ref{pr3} and \ref{mitlef} below). Thus, just like {in Theorem \ref{firsteh}}, we can  justify that  the main contribution of the SSF near a fixed Landau level ${\mathcal E}_j$ depends on the behavior of the corresponding band functions $E^N_j$, the associated eigenfunction, and the interplay of this objects at infinity (this  last one still given by  \eqref{bx1} below). The main difference  with the Dirichlet case comes from the fact that, at infinity  the band functions are {now} below the corresponding Landau level{. T}hen, up to a change of sign of $E_j^N- {\mathcal E}_j$ (or equivalently of $\lambda$ and of $V$), the above results remain true. More precisely, for $H^N_\pm:= H_0^N \pm V$, we have:

\begin{theorem}\label{NeuRob}  
The statements of Theorems \ref{firsteh} and \ref{V_power} hold true for the Neumann 
boundary conditions. 
Moreover{,} the results of Corollary \ref{coro1} have to be replaced by: {
\begin{enumerate} 
\item[{\rm 1.}] If $V$ satisfies \eqref{cd2}, then 
$
\ \xi(\mathcal{E}_j+\lambda;H^N_-,H^N_0) =O(1),$  as $\lambda \downarrow 0$.
\item[{\rm 2.}]  If  $V$ satisfies \eqref{12apr17}, then
$\xi({\mathcal E}_j- \lambda;H^N_+,H^N_0)=o (\lambda^{-2/m} ) $, as $\lambda \downarrow 0.$
\end{enumerate}
}

\end{theorem}

By the previous arguments it seems reasonable that these results could also be extended without mayor changes to other 2D magnetic models like{,} \eqref{Hk} in the  half-plane with a Robin boundary condition {and  some types of} Iwatsuka Hamiltonians \cite{iwa}.

\section{{Estimates on the SSF} }\label{SecThm1}
\subsection{Pushnitski's representation of the SSF}
{The proof of Theorem \ref{firsteh} is  based on the following  representation of the SSF}
  given by  A. Pushnitski in \cite{pu}. For $z \in \C^+$ and $V\geq 0$, define
$$ T(z):= V^{1/2} (H_0-z)^{-1} V^{1/2}.$$ As is  shown in {Lemma} \ref{propini} below,  the norm limits
$$
\lim_{\delta\downarrow 0}T(E+i\delta) = T(E+i0)
$$ exist for every $E \in \R\setminus\{ {\mathcal E}_j\}_{j=1}^\infty$, provided that $V$ satisfies \eqref{cd2}.
{Moreover}, $T(E+i0)$ is a Hilbert-Schmidt operator, and $0 \leq {\rm Im}\,T(E+i0)$.

\begin{theorem} \label{pushssf} {\rm \cite[Theorem 1.2]{pu}}
Assume that $V\geq 0$ satisfies \eqref{cd2}. Then
for almost all $E \in \R$ we have
    \bel{pssf}
    \xi(E;H_\pm,H_0) = \pm \frac{1}{\pi}\int_\R n_{\mp}(1;{\rm Re} \,T(E+i 0)+t\, {\rm Im} \, T(E +i0))\frac{dt}{1+t^2}.
    \ee
    \end{theorem}

\subsection{Spectral properties of $H_0$} \label{ss21}
We will  use \eqref{pssf}  to obtain information {about} the SSF at the threshold $\mathcal{E}_j$. This requieres to understand  the {behavior of the spectrum of  $H_0$ near $\mathcal{E}_j$}, which, in view of \eqref{sant1} and \eqref{oct9},   {are} intimately connected with   the behavior of the band function $E_j$ and the eigenprojection $\pi_j$
 at infinity. 

{We begin  then, with the definition and properties of the limit operator of $h(k)$}. Let $h_\infty(k)$ be the 
 self-adjoint realization in $L^2({\re})$ of
$$
 -\frac{d^2}{dx^2}
+ (bx-k)^2 , \quad k \in \R.
$$
The  operator $h_{\infty}(k)$ 
  has  discrete spectrum and its
eigenvalues are given by  $\mathcal{E}_j=b(2j-1), \, j \in \N$.
Denote by $\pi_{j, \infty}(k)$, $k \in \re$, $j \in {\mathbb N}$,
the orthogonal projection onto ${\rm Ker}\,(h_\infty(k) - \mathcal{E}_j)$.
Then, similarly to \eqref{lau5}
   $$
   \pi_{j,\infty}(k)  =\big| \psi_{j,\infty}(\cdot;k) \big\rangle \big\langle
\psi_{j,\infty}(\cdot;k)\big|,
   $$
    where
the eigenfunction $\psi_{j,\infty}(\cdot;k)$ satisfies
   $$
-\frac{\partial^2 \psi_{j,\infty}(x;k)}{\partial x^2} + (bx-k)^2
\psi_{j,\infty}(x;k) = \mathcal{E}_j \psi_{j,\infty}(x;k), \quad
\|\psi_{j,\infty}(\cdot;k)\|_{L^2(\re)} = 1.
   $$
    The function $\psi_{j, \infty}(\cdot;k)$ could  be chosen  real-valued.
Furthermore, the functions $\psi_{j,\infty}$ 
admit the following  explicit description. Namely, if we put
    \bel{61}
\varphi_j(x) : = \frac{1}{(\sqrt{\pi}(j-1)!\,2^{j-1})^{1/2}}
H_{j-1}(x) e^{-x^2/2}, \quad x \in \re, \quad j \in {\mathbb N},
    \ee
where
$$
{\rm H}_q(x): = (-1)^q e^{x^2} \frac{d^q}{dx^q} e^{-x^2}, \quad x
\in \re, \quad q \in {\mathbb Z}_+,
$$
are the Hermite polynomials (see e.g. \cite[Chapter I, Eqs. (8.5),
(8.7)]{bershu}). Then, the real-valued function $\varphi_j$ satisfies
$$
-\varphi_j''(x) + x^2 \varphi_j(x) = (2j-1) \varphi_j(x), \quad
\|\varphi_j\|_{L^2(\re)} = 1,
$$
 and hence
    \bel{sof9o}
\psi_{j,\infty}(x;k) = b^{1/4}\varphi_j(b^{1/2}x - b^{-1/2}k), \quad
j \in {\mathbb N}, \quad x \in \re, \quad k \in \re.\ee

Define the non-negative operator $$\Lambda_k:=h_\infty(k)^{-1}-({\mathbb{0}_-}-\oplus\, h(k)^{-1}),$$
where $\mathbb{0}_-$ is the zero operator in $L^2(-\infty,0)$.

\begin{proposition} \rm{\cite[{Propositions 3.4-3.5-3.6}]{bmr2}\label{pr1}}
For any  $l \in \N$,  as $ k \to \infty$:
\begin{enumerate}

\item[\rm{1.}]  $\|\Lambda_k\| = \frac{3\sqrt{b}}{2^{3/2}}k^{-2}(1+o(1)),$
\item [\rm{2.}]  $\|{\pi}_{j,\infty}(k)-(\mathbb{0}_-\,\oplus{\pi}_j(k))\|{}=O(\|{\pi}_{j,\infty}(k)\Lambda_k\|{}), 
    $
\item[\rm{3.}]   $
E_j(k)-\mathcal{E}_j={\mathcal E}_j^2\|{\pi}_{j,\infty}(k)\Lambda_k^{1/2}(k)\|^2\,(1+o(1)).
    $
    \end{enumerate}
\end{proposition}
\noindent As a consequence of Proposition \ref{pr1} we have:
\bel{bx1}
\lim_{k\to
\infty}(E_j(k)-\mathcal{E}_j)^{-1/2}\|{\pi}_{j,\infty}(k)-{(\mathbb{0}_-\oplus\pi}_j(k))\|{}=0.
    \ee

Denote by $\mathfrak{S}_p$, $p \in [1,\infty)$, the Schatten -- von Neumann class of compact operators, equipped with the norm
{$$
    \| T \|_p : = \Big(\sum_l s_l(T)^p  \Big)^{1/p},
   $$
   where $\{s_j(T)\}$ are the square root{s} of the eigenvalues of $T^*T$.}
   
   {For future references, it will be useful to have the following estimate of the difference of the eigenfunctions.} 
\begin{pr}\label{mitlef}
For any $j\in \N$ {it is satisfied}
$$
{\|\psi_j(\cdot;k)-\psi_{j,\infty}(\cdot;k)\|_{L^2(\R_+)}}=O\left(k^{-1}(E_j(k)-\mathcal{E}_j)^{1/2}\right), \quad k\to \infty.
$$
\end{pr}
\begin{proof}
Define $A(k):=\langle \psi_j(\cdot;k),\psi_{j,\infty}(\cdot;k)\rangle_{L^2(\R_+)}$, where $\langle\cdot,\cdot\rangle_{L^2(\re_+)}$ is the scalar product in $L^2(\re_+)$. Since $||(\mathbb{0}_-\oplus\pi_j(k))-\pi_{j,\infty}||_{2}^2=2(1-A(k)^2)$, we have that $A(k)^2\to 1$, as $k\to \infty$. {By} the continuity of $A(k)$ we may assume from the beginning that  $A(k)\to 1$, as $k\to \infty$.

Now, 
\bel{jun23_16}\ba{ll}&\displaystyle{\int_{\R_+}|\psi_j(x;k)-\psi_{j,\infty}(x;k)|^2\,dx}\\[1em]
=&\displaystyle{\int_{\R_+}\psi_j(x;k)^2\,dx-2\int_{\R_+}\psi_j(x;k)\psi_{j,\infty}(x;k)\,dx+\int_{\R_+}\psi_{j,\infty}(x;k)^2\,dx}\\[1em]
=&\displaystyle{2(1-A(k))-\int_{\R_-}\psi_{j,\infty}(x;k)^2\,dx}\\[1em]
=&\displaystyle{\frac{||(\mathbb{0}_-\oplus\pi_j(k))-\pi_{j,\infty}||^2_{2}}{1+A(k)}-\int_{\R_-}\psi_{j,\infty}(x;k)^2\,dx.}
\end{array}\ee
From the definition of $\psi_{j,\infty}$, straightforward calculations show that 
\bel{jun23_16_b} \int_{\R_-}\psi_{j,\infty}(x;k)^2\,dx=\frac{k^{2j-3}}{2b}e^{-b^{-1}k^2}\asymp k^{-2}(E_j(k)-{\mathcal E}_j), \quad k\to \infty.\ee
Where we have used \eqref{ivri_asymp} below. 
The statements   2. and 3. of Proposition \ref{pr1} impl{y that}
\bel{jun23_16_c} ||(\mathbb{0}_-\oplus\pi_j(k))-\pi_{j,\infty}||^2_{2}=O(||\pi_{j,\infty}(k)\Lambda_k||^2)=O(||\Lambda_k||(E_j(k)-{\mathcal E}_j)), \quad k\to \infty.\ee
Then, putting together \eqref{jun23_16},  \eqref{jun23_16_b}, \eqref{jun23_16_c}
and the  fact that $(1+A(k))^{-1} $ is bounded {for $k\to \infty$}, we finish the proof.
\end{proof}

\begin{proposition}\label{holder}
For any $j \in \N$, there exists $K_j\in\R$ such that $$ \|
\pi_j(k)-\pi_{j}(k')\|\leq C_j|k-k'|,$$ for a  constant $C_j$   independent of $k,k'\geq K_j$.
\end{proposition}
\begin{proof}
{Since $\pi_j$ is analytic, it is sufficient to prove that its derivative is uniformly bounded with respect to $k$ large enough. Set $K_j$ to be any number that satisfies   $E_{j-1}(K_j)<{\mathcal E}_j$.  By \eqref{oct9} and the strict monotonicity of the bands functions, we can take a contour $\gamma$  around  $E_j(k)$ such that no  other eigenvalues of $h(k)$  lie inside the region enclosed by $\gamma$ whenever $k\in (K_j,\infty)$. Then, for $k > K_j$, we have 
	$$\pi_j(k)=\frac{1}{2\pi i}\oint_\gamma (h(k)-\omega)^{-1}d\omega,$$
{and}
$$\pi^\prime_j(k)=\frac{-2}{2\pi i}\oint_\gamma (h(k)-\omega)^{-1} \, (bx-k) \,  (h(k)-\omega)^{-1} d\omega.$$
{To} conclude the proof 
{we need to show that}
 $(h(k)-\omega)^{-1}$ and $ (bx-k) \,  (h(k)-\omega)^{-1}$ are uniformly bounded with respect to $k\geq K_j$ and to $\omega \in \gamma$. 
{The first condition follows since  our choice of $K_j$} implies that $\hbox{dist}(\gamma, \cup_{k\geq K_j}E_j(k)) >0$. {For the second condition, we use} the following relations. First, {applying}  the spectral theorem, for any $u\in L^2(\R_+)$,  we have 
$$|\langle (h(k)-\omega)^{-1} u , u \rangle | \leq C \| u\|^2_{L^2(\R_+)}, \qquad C>0,$$
{and s}econd, using self-adjointness properties, we have
$$\langle (h(k)-\omega)^{-1} u , u \rangle =  
\langle (h(k)-\omega)^{-1} \Big( -\frac{d^2}{dx^2}
+ (bx-k)^2-\omega \Big) (h(k)-\omega)^{-1} u , u \rangle = $$
$$\|(bx-k)   (h(k)-\omega)^{-1} u\|^2_{L^2(\R_+)} + \|\frac{d}{dx} (h(k)-\omega)^{-1} u\|^2_{L^2(\R_+)} - \omega \|( (h(k)-\omega)^{-1} u\|^2_{L^2(\R_+)}, $$
}
 and consequently
$$\ba{ll}\|(bx-k)   (h(k)-\omega)^{-1} u\|^2_{L^2(\R_+)}  &\leq | \omega|  \|( (h(k)-\omega)^{-1} u\|^2_{L^2(\R_+)} + |  \langle (h(k)-\omega)^{-1} u , u \rangle |\\ &\leq C \| u \|^2_{L^2(\R_+)}.\ea$$
\end{proof}

\begin{pr}\label{pr3}\cite[Corollary 15.A.6]{ivrii2}
For any $j \in \N$, and $n\in \Z_+$ 
\bel{ivri_asymp} (E_j(k)-\mathcal{E}_j)^{(n)}\asymp k^{2j-1+n}e^{-b^{-1}k^2}, \quad k\to \infty. \ee
In addition 
\bel{asymband-}E_j(k)=k^2(1+o(1)), \quad k\to- \infty.\ee

\end{pr}
To finish this part we {define}  $\varrho_j:(0,\infty)\to \R$ as the inverse function of $E_j-\mathcal{E}_j$. Evidently,  {the function $\varrho_j$  is strictly decreasing with $\lim_{s\to \infty}\varrho_j(s)=-\infty$ and }$\lim_{s\downarrow0}\varrho_j(s)=\infty$. Moreover,   the preceding proposition yields: 
\bel{ivri_asymp_2}\ba{lll}\displaystyle{\varrho_j(s)\asymp|\ln s|^{1/2};}&
\displaystyle{\varrho'_j(s)\asymp\frac{-1}{s|\ln s|^{1/2}};}&
\displaystyle{\varrho''_j(s)\asymp\frac{1}{s^{2}|\ln s|^{1/2}},  \quad  s\downarrow 0.}\ea\ee

\subsection{Analysis of $T$ on the real axis} \label{s3}
{The following step in our proof {of} Theorem \ref{firsteh} consists in the analysis of the operator $T(E+i0)$ appearing in \eqref{pssf}.  We will decompose this operator according to the band structure of $H_0$. In Proposition \ref{fom} and its proof  we describe each one of the components of the decomposition.}

\begin{pr}\label{fom}
{Suppose that} $V\geq0$ satisfi{es} \eqref{cd2}. Then, for all $j \in \N$, and all  $z \in {\mathbb C}_+$, the operator $T_j(z)$ is in  $\mathfrak{S}_1$. Furthermore,
for any $E \neq \mathcal{E}_j$,
 the limit $$ \lim_{\delta \downarrow 0} T_j(E +i\delta):=T_j(E)$$ exists in the $\mathfrak{S}_1$-norm{,} and is continuous with respect to $E\in \R \setminus \{\mathcal{E}_j\}$ in the standard operator-norm. {Moreover}, for $E \neq {\mathcal E}_j$ it {is} satisfi{ed}
 \bel{8sep16}
 \|{\rm Re}(T_j(E))\|_1=O(|\mathcal{E}_j-E|^{-1}).
 \ee
\end{pr}
\begin{proof}
Define the operator-valued function $G_j:\R\to \mathfrak{S}_2(L^2(\R_+\times\R);\C)$, by

\bel{dec23_2}G_j(k)u:=\frac{1}{\sqrt{2\pi}} \int_{\R}\int_{\R_+}e^{-iky}V^{1/2}(x,y)\psi_j(x;k)u(x,y)\,dxdy, \quad u \in L^2(\R_+\times\R).\ee
For any fixed $z$ with Im $z>0$, let us prove that
\bel{12aug16}
||G_j^*(\cdot)G_j(\cdot)/(E_j(\cdot)-z)||_{1} \in L^1(\R).
\ee

First note that the function $G_j^*G_j$ 
is locally Lipschitz and then measurable. Moreover, 
\bel{29jul16}\|G_j(k)^*G_j(k)-G_j(k')^*G_j(k')\|_{1}\leq C_j |k-k'|, \ee
where $C_j>0$ and is independent of $k,k'$  big enough. {Using that $G_j(k)^*G_j(k)$ is a finite rank operator,} these  properties follow from 
\bel{oct2}\ba{ll}
&\|G_j(k)^*G_j(k)-G_j(k')^*G_j(k')\|^2_{2}\\[.5em]
\leq &\displaystyle{\frac{1}{2\pi}\left(\sup_{x \in R_+}\int_\R V(x,y)\,dy\right)^2\int_{\R_+}|\psi_j(x;k)\psi_j(x';k)-\psi_{j}(x;k')
\psi_{j}(x';k')|^2\,dxdx'}\\[1em]
=&\displaystyle{\frac{1}{2\pi}\left(\sup_{x \in R_+}\int_\R V(x,y)\,dy\right)^2||\pi_j(k)-\pi_{j}(k')||^2_{2}}.
\ea\ee
The Lipschitz condition follows from \eqref{oct2} and the fact that the projection $\pi_j$ depends analytically  on $k$. {Inequality} \eqref{29jul16} follows from Proposition \ref{holder}.


Next, set $G_{j,\infty}:\R\to
\mathfrak{S}_2(L^2(\R_+\times\R);\C)$  as the operator valued function given by
\bel{nov19}G_{j,\infty}(k)u:=\frac{1}{\sqrt{2\pi}} \int_{\R}\int_{\R_+}e^{-iky}V^{1/2}(x,y)\psi_{j,\infty}(x;k)u(x,y)\,dxdy, \quad u \in L^2(\R_+\times\R).\ee
The function $\psi_{j,\infty}$ being defined in \eqref{sof9o}.

Obviously, for any $k \in \R$ the operators  $G_j(k)^*G_j(k),\, {G_{j,\infty}(k)^*G_{j,\infty}(k)}
:L^2(\R_+\times\R) \to L^2(\R_+\times\R)$ are of rank one, hence  the{ir} Hilbert-Schmidt and trace norms coincide. {Based on this property we will  use} the estimate
\bel{ld6}
\frac{||G_j(k)^*G_j(k)||_{1}}{|E_j(k)-z|} \leq
\frac{||G_{j,\infty}(k)^*G_{j,\infty}(k)||_{2}}{|E_j(k)-z|}+
\frac{ ||G_j(k)^*G_j(k)-G_{j,\infty}(k)^*G_{j,\infty}(k)||_{2}}{|E_j(k)-z|}.
\ee
For the first term on the r.h.s. of \eqref{ld6}, we can see that \eqref{cd2} implies
\bel{l12}\ba{ll}
&||G_{j,\infty}(k)^*G_{j,\infty}(k)||^2_{2}\\[.7em]
=&\displaystyle{\frac{1}{2\pi}\int_{\R^2}\int_{\R_+^2} V(x,y)V(x',y')\; |\psi_{j,\infty}(x;k)\psi_{j,\infty}(x';k)|^2\,dxdx'\,dydy'}\\[1em]
\leq&\displaystyle{\frac{1}{2\pi}\left(\int_\R\langle y\rangle
^{-m/2}\,dy\right)^2\left(\int_{\R}\langle
x\rangle^{-m/2}\psi_{j,\infty}(x;k)^2\,dx\right)^2}.
\ea\ee

Now, from {\eqref{sof9o}}
$$\int_{\R} \langle
x\rangle^{-m/2}\psi_{j,\infty}(x;k)^2dx=b^{1/4}\left(\varphi_j^2(b^{1/2}\,\cdot)*\langle
\,\cdot\,\rangle^{-m/2}\right){{(b^{-1}k)}},$$ 
with 
$\varphi_j$ {being} introduced in \eqref {61}. Thus, Young's
inequality together with $m>2$ imply that
${||G_{j,\infty}(\cdot)^*G_{j,\infty}(\cdot)||}_{2} \in L^{1}(\R)$. Using
that $(E_j(k)-z)^{-1}$ is bounded {for ${\rm Im} z >0$}, we get the integrability of the first term in the r.h.s. of \eqref{ld6}.

The second term  of the r.h.s. of \eqref{ld6} is shown to be integrable in the
following way. For $k$ {large and {non-}positive} we use the integrability of
$(E_j(k)-z)^{-1}$ (see \eqref{asymband-}) and the bounds
$$||G_j(k)^*G_j(k)||_{2},\, ||G_{j,\infty}(k)^*G_{j,\infty}(k)||_{2}\leq \frac{1}{\sqrt{2\pi}}\sup_{x\in\R_+}\int_\R V(x,y) \,dy.$$
On the other side, {arguing as in \eqref{oct2}, \eqref{l12}} 
\bel{ld1}\ba{ll}
&||G_j(k)^*G_j(k)-G_{j,\infty}(k)^*G_{j,\infty}(k)||^2_{2}\\[.7em]
\leq&\displaystyle{\frac{1}{2\pi}\left(\sup_{x \in \R_+}\int_\R V(x,y)\,dy\right)^2||(\mathbb{0}_-\oplus\pi_j(k))-\pi_{j,\infty}(k)||^2_{2}}.
\ea\ee
Then, by {part} 2. of Proposition \ref{pr1}
$$
||G_j(k)^*G_j(k)-G_{j,\infty}(k)^*G_{j,\infty}(k)||_{2}\leq C||\Lambda_k {\pi}_{j,\infty}(k)||,$$
for all sufficiently large $k$, which together with {part} 3. of Proposition \ref{pr1}, and Proposition \ref{pr3} imply
 the integrability of
$||G_j(k)^*G_j(k)-G_{j,\infty}(k)^*G_{j,\infty}(k)||_{2}$ at
infinity.

As a  consequence of \eqref{12aug16}, the operator
$\int_{\R} \frac{G_{j}(k)^*G_{j}(k)}{E_j(k)-z}dk$ is well defined {and belongs to}  $\mathfrak{S}_1$, where the integral is understood in
the Bochner sense, and further,  it is easy to see that  for ${\rm Im} \,z>0$
$$T_j(z)=\int_{\R} \frac{G_{j}(k)^*G_{j}(k)}{E_j(k)-z}\,dk.$$
{This implies that $T_j(z)$ is in $\mathfrak{S}_1$.}

Next, suppose that $z=\mathcal{E}_j+\lambda +i\delta$, with
$\lambda, \delta \in \R$. Recall that  $\varrho_j$ is the inverse function of
 $E_j-\mathcal{E}_j$. Hence, making the change of variables $E_j(k)-\mathcal{E}_j \mapsto s$,
we get \bel{l13}T_j(\mathcal{E}_j+\lambda
+i\delta) =\int_{-\infty}^{\infty}\frac{G_{j}(k)^*G_{j}(k)}{E_j(k)-\mathcal{E}_j-\lambda
 -i\delta}\,dk =\int_0^\infty
\frac{G_{j}(\varrho_j(s))^*G_{j}(\varrho_j(s))}{\lambda-s
+i\delta}\,\varrho_j'(s)\, ds.\ee

If $\lambda <0$,  by Lebesgue dominated convergence theorem
\bel{ld14}\ba{rcl}\displaystyle{
\lim_{\delta \downarrow 0} T_j(\mathcal{E}_j+\lambda +i\delta)}&=&\displaystyle{\int_{-\infty}^{\infty}\frac{G_{j}(k)^*G_{j}(k)}{E_j(k)-\mathcal{E}_j-\lambda}\,dk}\\[1em]
=\displaystyle{\int_0^\infty
\frac{G_{j}(\varrho_j(s))^*G_{j}(\varrho_j(s))}{\lambda-s}\varrho_j'(s)\,ds}&=&\displaystyle{V^{1/2}{\mathcal F}^*\int_\R^\oplus\frac{\pi_j(k)}{E_j(k)-\mathcal{E}_j-\lambda}\,dk\,{\mathcal F}V^{1/2}},
\ea\ee
in the $\mathfrak{S}_1$ norm.

If $\lambda >0$, write
$$\ba{ll}&\displaystyle{\int_0^\infty \frac{G_{j}(\varrho_j(s))^*G_{j}(\varrho_j(s))}{\lambda-s +i\delta}\varrho_j'(s)\,ds}\\[1em]
=&\displaystyle{\int_0^\infty \frac{G_{j}(\varrho_j(s))^*G_{j}(\varrho_j(s))}{{(\lambda-s)^2 +
\delta^2}}(\lambda-s)\varrho_j'(s)\,ds-i\delta\int_0^\infty
\frac{G_{j}(\varrho_j(s))^*G_{j}(\varrho_j(s))}{(\lambda-s)^2+\delta^2}\varrho_j'(s)}\,ds.
\ea$$

Then, it is easy to see that
$$\lim_{\delta \downarrow 0} i\delta\int_0^\infty \frac{G_{j}(\varrho_j(s))^*G_{j}(\varrho_j(s))}{(\lambda-s)^2+\delta^2}\varrho_j'(s) \,ds=
i \pi G_j(\varrho_j(\lambda))^*
G_j(\varrho_j(\lambda))\varrho_j'(\lambda).$$ 
{On the other side}, using  the analyticity
of $\varrho_j$ as well as  \eqref{29jul16}, we can
prove that
\bel{8sep16a}\ba{ll}&\displaystyle{\lim_{\delta \downarrow 0 }\int_0^\infty \frac{G_{j}(\varrho_j(s))^*G_{j}(\varrho_j(s))}{(\lambda-s)^2+\delta^2}(\lambda-s)\varrho_j'(s)\,ds=
\mbox{p.v.}\int_0^\infty
\frac{G_{j}(\varrho_j(s))^*G_{j}(\varrho_j(s))}{(\lambda-s)}\varrho_j'(s)\,ds}\\[1em]
=&\displaystyle{\int_0^{\lambda-\epsilon_\lambda}\frac{G_{j}(\varrho_j(s))^*G_{j}(\varrho_j(s))}{(\lambda-s)}\varrho_j'(s)\,ds+
\int_{\lambda+\epsilon_\lambda}^\infty\frac{G_{j}(\varrho_j(s))^*G_{j}(\varrho_j(s))}{(\lambda-s)}\varrho_j'(s)\,ds}\\[1em]
+&\displaystyle{\int_0^{\epsilon_\lambda}G_{j}(\varrho_j(\lambda+s))^*G_{j}(\varrho_j(\lambda+s))\varrho_j'(\lambda+s)-
{G_{j}(\varrho_j(\lambda-s))^*G_{j}(\varrho_j(\lambda-s))}\varrho_j'(\lambda-s)\frac{ds}{s}},\ea\ee
{for any $\epsilon_\lambda\in(0,\lambda).$}

Thus, for $\lambda>0$ 
 \bel{ld15} T_j({\mathcal E}_j+\lambda)=\mbox{p.v.}\int_0^\infty
\frac{G_{j}(\varrho_j(s))^*G_{j}(\varrho_j(s))}{(\lambda-s)}\varrho_j'(s)\,ds
-i \pi G_j(\varrho_j(\lambda))^*
G_j(\varrho_j(\lambda))\varrho_j'(\lambda).\ee

It is obvious that  for  $\lambda <0$
\bel{ld7}
T_j(\mathcal{E}_j+\lambda)^*=T_j(\mathcal{E}_j+\lambda)= {\int_0^\infty \frac{G_{j}(\varrho_j(s))^*G_{j}(\varrho_j(s))}{(\lambda-s)}\varrho_j'(s)\,ds }
,\ee and  for
$\lambda >0$
\bel{ld16}\begin{array}{l}\mbox{Re}\,T_j({\mathcal E}_j+\lambda)=\mbox{p.v.}\displaystyle{\int_0^\infty \frac{G_{j}(\varrho_j(s))^*G_{j}(\varrho_j(s))}{(\lambda-s)}\varrho_j'(s)\,ds}\\[1em]
\mbox{Im}\,T_j({\mathcal E}_j+\lambda)=-\pi
G_j(\varrho_j(\lambda))^*
G_j(\varrho_j(\lambda))\varrho_j'(\lambda).\end{array}\ee
Then, the continuity property of T$_j$ follows immediately from  { \eqref{ld14}, }
\eqref{ld15}, the continuity of $\varrho_j$ and  the continuity of $G_{j}^*G_{j}$.

Finally,  we prove \eqref{8sep16}.  For  $\lambda>0${, we} take into account  \eqref{ld16}, \eqref{8sep16a} {with $\epsilon_\lambda=\lambda/2$}, the inequalities  
$$\left\|\int_0^{\lambda/2}\frac{G_{j}(\varrho_j(s))^*G_{j}(\varrho_j(s))}{(\lambda-s)}\varrho_j'(s)\,ds
\right\|_1\leq \frac{2}{\lambda}\int_{\varrho_j(\lambda/2)}^\infty\left\|G_j(k)^*G_j(k)\right\|_1\,dk=O(\lambda^{-1}),$$
$$\left\|\int_{3\lambda/2}^\infty\frac{G_{j}(\varrho_j(s))^*G_{j}(\varrho_j(s))}{(\lambda-s)}\varrho_j'(s)\,ds
\right\|_1\leq {\frac{{C}}{\lambda}}\int^{\varrho_j(3\lambda/2)}_{0}\left\|G_j(k)^*G_j(k)\right\|_1\,dk {\,+\,O(1)}=O(\lambda^{-1}),$$
and, as a special case in  the proof of Lemma \ref{estiPV} below:
$$\left\|{\rm p.v.}\int_{\lambda/2}^{3\lambda/2}\frac{G_{j}(\varrho_j(s))^*G_{j}(\varrho_j(s))\varrho_j'(s)}{s-\lambda}ds\right\|_1=O\left(\frac{\lambda^{-1}}{|\ln \lambda|^{1/2}}\right).$$
For $\lambda<0$, from { \eqref{asymband-}}, \eqref{ld7} and  \eqref{ld14}  
$$\|{\rm Re}\, T_j(\lambda)\|_1\leq C\,\sup_{k\in\R_+}|E_j(k)-\mathcal{E}_j-\lambda|^{-1}\int_{0}^\infty\|G_j(k)^*G_j(k)\|_1\,dk  \,+O(1)=O(\lambda^{-1}).$$

\end{proof}

For $j\in \N$, {let us introduce the projector}
$$P_j^+:={\mathcal F}^*\int_\R^\oplus	\sum_{l>j}\pi_l(k)\,dk\,{\mathcal F},$$
and for $z\in{ \C_+}$
{
$$T_j^+(z):=V^{1/2}P_j^+(H_0-z)^{-1}P_j^+V^{1/2}=V^{1/2}{\mathcal F}^*\int_\R^\oplus	\sum_{l>j}(E_l(k)-z)^{-1}\pi_l(k)\,dk\,{\mathcal F}V^{1/2}.$$}
In both cases the infinite sums being understood in the strong sense.
\begin{lemma}\label{posipart}
Let $j \in \N$. Then, for $E\in(-\infty,{\mathcal E}_{j+1})$ the limit  
    \bel{sof11}
\lim_{\delta \downarrow 0} T^+_j(E+i\delta)=:T_j^+(E)
    \ee
    exists in the norm sense {and is self-adjoint}. Moreover, the function $T_j^+:\overline{\C_+}\setminus[{\mathcal E}_{j+1},\infty)\to {\mathfrak S}_2$ is continuous, and for $E\in(-\infty,{\mathcal E}_{j+1})$
\bel{8sep16b}\|T_j^+(E)\|_{{\mathfrak S}_2}\leq C\,{\mathcal E}_{j+1}({\mathcal E}_{j+1}-E)^{-1},\ee   
 where the constant $C$ is independent of $E$.   
\end{lemma}
\begin{proof} 
Due to the band structure \eqref{sant1},  the spectrum of  $P_j^+H_0P_j^+$ is $[\mathcal{E}_{j+1},\infty)$, and then
the operator valued function $T_j^+(\cdot)$ is analytic in $\overline{\C_+}\setminus[{\mathcal E}_{j+1},\infty)$. This implies in particular the existence of the limit \eqref{sof11} {and the self-adjointness of $T_j^+(E)$}.

{In the  proof of  \eqref{8sep16b}} for  $E\in(-\infty,{\mathcal E}_{j+1})$,  {since $H_0^{-1}V^{1/2} \in \mathfrak{S}_2$}, we use 
$$\ba{rl}\|T_j^+(E)\|_{2}=\| V^{1/2}P_j^+(H_0-E)^{-1}V^{1/2}\|_{2}&\leq \| V^{1/2}P_j^+(H_0-E)^{-1}H_0\|\|H_0^{-1}V^{1/2}\|_{2}\\[.5em]
&\displaystyle{\leq    \| V^{1/2}\| \|H_0^{-1}V^{1/2}\|_{2}\frac{\mathcal{E}_{j+1}}{{\mathcal E}_{j+1}-E}.}\ea
$$
In the same {manner} it can be shown that  for $E_1, E_2 \in (-\infty,{\mathcal E}_{j+1})$
$$\|T_j^+(E_1)-T_j^+(E_2)\|_{2}\leq |E_1-E_2| \| V^{1/2}\| \|H_0^{-1}V^{1/2}\|_{2} \frac{\mathcal{E}_{j+1}}{({\mathcal E}_{j+1}-E_1)({\mathcal E}_{j+1}-E_2)},$$
which implies the continuity of $T_j^+$ in the Hilbert-Schmidt norm.
\end{proof}

\begin{lemma}\label{propini}
Let $z= E+i \delta \in {\mathbb C}_+$. Then for all
$E \in \R\setminus \{{\mathcal E}_l\}_{l\in \N}$ the norm limit $\lim _{\delta
\downarrow 0}T(E +i\delta)=T(E +i0)$ exists  {in the Hilbert-Schmidt  class}. If for some $j \in \N$, $E \in
(\mathcal{E}_{j-1},\mathcal{E}_j)\cup (\mathcal{E}_j,{\mathcal
E}_{j+1})$,  then
\bel{28mar17}
T(E+i0)=T_j^-(E)+T_j(E)+T_j^+(E),
\ee
where $T^-_j(E)=\sum_{l=1}^{j-1}T_l(E)$. Moreover,
\bel{prd31}
{\rm Re}\,T(E+i0)= {\rm Re}\,T_j^-(E)+{\rm Re}\,T_j(E)+T_j^+(E)
\ee
\bel{prd32}
{\rm Im}\,T(E+i0)={\rm Im}\,T_j^-(E)+{\rm Im}\,T_j(E),
\ee
and we have  continuous dependence on $\R \setminus\{{\mathcal E}_l; l\in \N \}$, of \,${\rm Re}\,T(\cdot +i0)$ with the standard operator norm, and of ${\rm Im}\,T(\cdot + i0)$ in the trace class norm.
\end{lemma}

\begin{proof}
{For the proof of the existence of the limit} it suffices to use the representation
$$\ba{ll}
T(E + i\delta) &=\displaystyle{V^{1/2}{\mathcal F}^* \int_\R^\oplus\sum_{l\in \N}(E_l(k)-(E +i \delta))^{-1}\pi_l(k)\,dk\,{\mathcal F} V^{1/2}}\\[.5em]
&={\displaystyle{\sum_{l=1}^{j-1}T_l(E+i\delta)+T_j(E+i\delta)+T_j^+(E+i\delta)}},\ea
$$
and  apply Proposition \ref{fom} together with  Lemma \ref{posipart} {to obtain \eqref{28mar17}}. Next,  if $E  \in  (\mathcal{E}_j,{\mathcal E}_{j+1})$,  $T_j^+(E)$ is self-adjoint and   then from \eqref{ld16}, the imaginary part {of $T(E+ i0)$} is just  the finite rank operator
\bel{28apr17a}{\rm Im}\,T(E+ i0) = -\pi \sum_{l \leq j} G_l(\varrho_l(E-\mathcal{E}_l))^* G_l(\varrho_l(E-\mathcal{E}_l))\varrho_l'(E-\mathcal{E}_l),\ee
{which is {obviously} continuous in the trace norm. The continuity of   ${\rm Re}\,T(E+ i0)$ follows from Lemma \ref{posipart} and {Proposition \ref{fom}}. The case $E  \in  (\mathcal{E}_{j-1},{\mathcal E}_{j})$ is similar}. 

\end{proof}

{\begin{remark}
The decomposition used in the above results was inspired by \cite{bsoccrai}. Related tools appear also in \cite{popsoc} and \cite{hispopsoc}.
\end{remark}

{\begin{remark}
{Following the arguments of   \cite[Proposition 2.5]{bpr}, {and using} Lemma \ref{propini}, it should be possible to prove that  $\xi(\cdot ;H,H_0)$ is continuous on $\R \setminus (\sigma_p(H)\cup \{ {\mathcal E}_l, \, l \in \N\})$.}
\end{remark}

\subsection{{Proof of Theorem \ref{firsteh}}} \label{s4}
{Let us first recall some well know results for the counting function ({defined by \eqref{defn}})  that  will be repeatedly used henceforth.} For $r_1, r_2>0$,  we have  the  Weyl inequalities
    \bel{lau11}
    n_\pm(r_1 + r_2; A_1 + A_2) \leq n_\pm(r_1; A_1) + n_\pm(r_2; A_2),
    \ee
where  $A_j$, $j=1,2$, are
linear self-adjoint compact operators (see e.g. \cite[Theorem 9.2.9]{birsol}). 
For a compact  not necessarily self-adjoint operator $A$ set 
$$
n_*(s;A) : = n_+(s^2; A^* A), \quad s>0.
$$
Obviously $n_*(r;A)=n_*(r;A^*)$, 
 and if $A=A^*$
\bel{oct6_2}
n_*(r;A)=n_+(r;A)+n_-(r;A).
\ee
Further,  the Ky Fan inequality {states that} for $r_1, r_2>0$    
\bel{lau13}
    n_*(r_1 + r_2; A_1 + A_2) \leq n_*(r_1; A_1) + n_*(r_2;
A_2),
    \ee
for $A_j$, $j=1,2$,  not necessarily self-adjoint (see e.g. \cite[Subsection 11.1.3]{birsol}).

Finally,  for  $A$ in the Schatten -- von Neumann class  $\mathfrak{S}_p$,  we have  the  Chebyshev-type estimate
   \bel{dj36}
    n_*(r; A) \leq r^{-p} \|A\|_p^p,
    \ee
    for any $r > 0$ and $p \in [1, \infty)$.

{Now, to begin the proof of the Theorem, we note that from} \eqref{pssf} it is necessary to estimate
$n_\pm(1;{\rm Re}\, T(\mathcal{E}_j+\lambda+i0)+t\, \mbox{Im}
\,T(\mathcal{E}_j+\lambda+i0))$.  
Since for $\lambda$ small, $\mbox{Im}\,T(\mathcal{E}_j+\lambda+i0)$ is  of rank $j$ {(see \eqref{28apr17a})},  we have the
inequalities
\bel{prd21}\ba{lll}
 n_\pm(1;\mbox{Re} \,T(\mathcal{E}_j+\lambda +i0)) -j &\leq & n_\pm(1;\mbox{Re} \,T(\mathcal{E}_j+\lambda+i0)+t\, \mbox{Im} \,T(\mathcal{E}_j+\lambda+i0)) \qquad
\\[.5em]
 &\leq & n_\pm(1;\mbox{Re}\, T(\mathcal{E}_j+\lambda+i0 ))+j,
\ea\ee
for all $t \in \R$.

Next, for the real part,   Lemma
\ref{posipart} implies that {for $\lambda\to  0,$} the norm limit of the operator $T_j^+(\mathcal{E}_j+\lambda)$ is just {$T_j^+(\mathcal{E}_j)$}.  Therefore, {from \eqref{lau11},} for all $r>0$ 
$$
n_\pm(r;T_j^+(\mathcal{E}_j+\lambda)){\leq n_\pm(r/2;T_j^+(\mathcal{E}_j))}=O(1), \quad \lambda \to 0,
$$
and \eqref{prd31} together with  Weyl inequalities \eqref{lau11}  imply {that  for any $r \in (0,1)$}
$$
n_\pm(1+r;\mbox{Re}\,T_j^-(\mathcal{E}_j+\lambda)+\mbox{Re}\,T_j(\mathcal{E}_j+\lambda))+O(1) \leq
$$
$$
n_\pm(1;\mbox{Re}\, T(\mathcal{E}_j+\lambda +i0))\leq
$$
    \bel{prd22}
n_\pm(1-r;\mbox{Re}\,T_j^-(\mathcal{E}_j+\lambda)+\mbox{Re}\,T_j(\mathcal{E}_j+\lambda))+O(1),
\quad \lambda \to 0.    \ee
 To finish the proof we need to show that
$n_\pm(r;\mbox{Re}\,T_j^-(\mathcal{E}_j+\lambda){{)}}$ remains bounded for all
$r>0$,  
as $\lambda \to 0$. From  Proposition \ref{fom}, the  Lebesgue dominated convergence
theorem implies that 
 \bel{sof10}\ba{lll}
\lim_{\lambda \to 0}\mbox{Re}\,T_l(\mathcal{E}_j+\lambda)&=&\displaystyle{\lim_{\lambda \to 0}\mbox{p.v.}\int_0^\infty
 \frac{G_{l}(\varrho_l(s))^*G_{l}(\varrho_l(s))}{s-2b(j-l)-\lambda}\varrho_l'(s)\,ds}\\[1em]
 &=&\displaystyle{\mbox{p.v.}\int_0^\infty
\frac{G_{l}(\varrho_l(s))^*G_{l}(\varrho_l(s))}{s-2b(j-l)}\varrho_l'(s)\,ds}.
    \ea\ee
Then, {since}
${\rm{Re}}\,{T_j^-}(\mathcal{E}_j+\lambda)=\sum_{l=1}^{j-1}{\rm{Re}}\,T_l(\mathcal{E}_j+\lambda)$:  
{\bel{18may17}
n_\pm(r;\mbox{Re}\,T_j^-(\mathcal{E}_j+\lambda){{)}}=O(1), \quad \lambda \to 0.\ee}
Using \eqref{lau11} and putting together \eqref{pssf}, \eqref{prd21}, \eqref{prd22},  and \eqref{sof10}, \eqref{18may17}, we
obtain \eqref{pr41}.

\subsection*{Proof of Corollary \ref{cor1}}
Suppose that {the compact set}  $\mathcal{C}$ is contained in $[0,{\mathcal E}_{l_0}]{\setminus \{{\mathcal E}_{l}, \, 1\leq l \leq l_0 \}}$, for some $l_0\in\N$. If $E\in {\mathcal C}$, by Theorem  \ref{pushssf} and \eqref{prd21} 
$$|\xi(E;H_\pm,H_0)|\leq {\ n_\mp(1;{\rm Re}\,T(E+i0))+l_0.}$$
From \eqref{prd31}, \eqref{oct6_2}  and \eqref{dj36}
$$\ n_\mp(1;{\rm Re}\,T(E+i0)) 
{\leq 2 \sum_{l=1}^{l_0}\|{\rm Re}\, T_l(E)\|_1+ 4 \|T_{l_0}^+(E)\|^2_{2 }.}$$
Then, {Lemma \ref{posipart} and Proposition \ref{fom}}
give us the boundedness  of the SSF on $\mathcal{C}$.

\section{{Asymptotic behavior of the SSF}}\label{SecThm2}
The proof of   \eqref{the4} is based on 
{Theorem \ref{firsteh}} for both cases, ``+'' and ``-''. However, the ``+'' case present{s} more difficulties  due to the presence of the principal value term  in  its effective Hamiltonian (compare  {\eqref{ld7}}
with \eqref{ld16}). In order to maintain the notation simple and to avoid  repetitions  of arguments we will {set forth} the proof of \eqref{the4}  only in the ``+'' case i.e.,  for $V\geq 0$ and $\lambda >0$ we will study  $\xi({\mathcal E}_j+\lambda;H_+,H_0)$.  The proof of the ``-'' case can be performed using the same kind of ideas.

{We already know an explicit form of the operator ${\rm Re}\,T_j(\mathcal{E}_j+\lambda)$, {which is given by $\mbox{p.v.}\int_0^\infty\frac{G_{j}(\varrho_j(s))^*G_{j}(\varrho_j(s))}{(\lambda-s)}\varrho_j'(s)\,ds$.   Our next task,  is to estimate the contribution that each part of the  decomposition  \eqref{8sep16a}  makes} to the behavior of the counting function $n_-(r; {\rm Re}\, T_j(\mathcal{E}_j+\lambda))$.  }

{First,} for   $(x,y) \in \R^2$ define
$$\widetilde{V}(x,y) : =\left\{
              \begin{array}{ll}
                0 \, &\mbox{if}\quad x\leq 0, \\
                V(x,y) \,& \hbox{if} \quad x >0.
              \end{array}
            \right.$$
It is clear that the operator  in $L^2(\R^2)$,
\bel{dic11_2} \widetilde{V}^{1/2}{\mathcal
F}^*\, \int_{\R}^\oplus
(E_j(k)-z)^{-1}(\mathbb{0}_-\oplus\pi_{j}(k))\,dk\,\mathcal{F}\widetilde{V}^{1/2}\ee
has the same non zero eigenvalues that $T_j( z)= V^{1/2}{\mathcal
F}^*\, \int_{\R}^\oplus
(E_j(k)-z)^{-1}\pi_{j}(k)\,dk\,\mathcal{F}V^{1/2},$
which  acts  in $L^2(\R_+\times \R)$. Moreover, all the   results of Section \ref{SecThm1} for $T_j(z)$  are valid for the operator \eqref{dic11_2} with obvious modifications.  In what follows we will consider the operator defined by \eqref{dic11_2}, and its  corresponding limit as $z$ {tends to} $\mathcal{E}_j+\lambda$, but by an abuse of notation we will  denote them by $T_j(z)$ and $T_j(\mathcal{E}_j+\lambda)$ as well.

Now, for {$\lambda \in \R$} and $I \subset (0,\infty) \setminus \{\lambda\} $  let us introduce ${S}_{j}[\lambda, I]:L^2({\varrho_j(I)})\to L^2(\R^2)$ as the integral operator with kernel
\bel{nov5_2}
(2\pi)^{-1/2} \widetilde{V}(x,y) ^{1/2} e^{iky}\psi_{j}(x;k)|E_j(k)-{\mathcal E}_j-\lambda|^{-1/2}.
\ee
In   the same {manner},  define  $S_{j,\infty}[\lambda, I]:L^2({\varrho_j(I)})\to L^2(\R^2)$ as the integral operator with kernel \eqref{nov5_2}, but using $\psi_{j,\infty}$ instead of $\psi_{j}$. 

{Note that 
$${S_jS_j^*[\lambda, I]}:=S_j[\lambda, I]S_j[\lambda, I]^*=\int_{I}
\frac{G_{j}(\varrho_j(s))^*G_{j}(\varrho_j(s))}{{\lambda-s}}\varrho_j'(s)\,ds.
$$
Then, it is clear that {for any $\epsilon_\lambda \in  (0,\lambda) $}
\bel{decomp+}\begin{array}{l}
{\rm Re}\,T_j({\mathcal E}_j+\lambda)=\mbox{p.v.}\displaystyle{\int_{\lambda -\epsilon_\lambda }^{\lambda+ \epsilon_\lambda}
\frac{G_{j}(\varrho_j(s))^*G_{j}(\varrho_j(s))}{{\lambda-s}}\varrho_j'(s)ds}\\[1em]
+ S_{j} S_{j} ^* [\lambda, (\lambda+ \epsilon_\lambda ,  \infty) ] - S_{j} S_{j} ^* [\lambda, (0, \lambda-\epsilon_\lambda ) ],
\\
\end{array}\ee
{and in order to estimate $n_\pm(r; {\rm Re\,}T_j(\mathcal{E}_j+\lambda))$, we will have to control each term of \eqref{decomp+}.}

\subsection{First spectral estimates}\label{ssAbove}

\begin{lemma}\label{estiPV}
Fix  $r>0$.   Then, for  $\epsilon_\lambda \in (0,\lambda)$, as $\lambda$ tends to $0$
\bel{nov11_3}n_\pm\left(r;{\rm p.v.}\displaystyle{\int_{\lambda-\epsilon_\lambda}^{\lambda+\epsilon_\lambda}
\frac{G_{j}(\varrho_j(s))^*G_{j}(\varrho_j(s))}{{\lambda-s}}\varrho_j'(s)\,ds}\right)=
O\Big(\frac{\epsilon_\lambda}{\lambda^2 |\ln \lambda|^{\frac12}}\Big).\ee
\end{lemma}
\begin{proof}
To prove \eqref{nov11_3}  write
$$
\ba{ll}
&\mbox{p.v.}\displaystyle{\int_{\lambda-\epsilon_\lambda}^{\lambda+\epsilon_\lambda}\frac{G_{j}(\varrho_j(s))^*G_{j}(\varrho_j(s))}{{\lambda-s}}\varrho_j'(s)\,ds}\\[1em]
=&\displaystyle{\int_{0}^{\epsilon_\lambda}
\left({G_{j}(\varrho_j(\lambda+s))^*G_{j}(\varrho_j(\lambda+s))-G_{j}(\varrho_j(\lambda-s))^*G_{j}(\varrho_j(\lambda-s))}\right)\varrho'_j(\lambda+s) \frac{ds}{s}}\\[1em]
&+\displaystyle{\int_{0}^{\epsilon_\lambda}
{G_{j}(\varrho_j(\lambda-s))^*G_{j}(\varrho_j(\lambda-s))\left(\varrho_j'(\lambda+s)-\varrho_j'(\lambda-s)\right)}\frac{ds}{s}}\\[1em]
=&\mathcal{M}_1(\lambda)+\mathcal{M}_2(\lambda).
\ea
$$
Using \eqref{lau11}, \eqref{oct6_2} and  \eqref{dj36} we get
$$
n_\pm\left(r;\mbox{p.v.}\displaystyle{\int_{\lambda-\epsilon_\lambda}^{\lambda+\epsilon_\lambda}
\frac{G_{j}(\varrho_j(s))^*G_{j}(\varrho_j(s))}{{\lambda-s}}\varrho_j'(s)\,ds}\right)
\leq \frac{1}{2r} \|\mathcal{M}_1(\lambda)\|_{1}+\frac{1}{2r} \|\mathcal{M}_2(\lambda)\|_{1}.
$$
From  \eqref{29jul16} and \eqref{ivri_asymp_2}
$$
\ba{ll}
&\|\mathcal{M}_1(\lambda)\|_{1}\\[.5em]
\leq& \displaystyle{\int_{0}^{\epsilon_\lambda}
\left\|{G_{j}(\varrho_j(\lambda+s))^*G_{j}(\varrho_j(\lambda+s))-G_{j}(\varrho_j(\lambda-s))^*G_{j}(\varrho_j(\lambda-s))}\right\|_1|\varrho'_j(\lambda+s)| \frac{ds}{s}}\\[1em]
\leq &C \displaystyle{\int_{0}^{\epsilon_\lambda}\left|\frac{\varrho_j(\lambda+s)-\varrho_j(\lambda-s)}{s}\right|\left|\varrho'_j(\lambda+s)\right| ds}\\[1em]
\leq& C \displaystyle{\int_{0}^{\epsilon_\lambda}\frac{ds}{\lambda^2|\ln \lambda|}}=
O\left(\frac{\epsilon_\lambda}{\lambda^2 |\ln \lambda|}\right)
, \quad \lambda \downarrow 0.
\ea
$$
By a similar procedure we can see that
 $$\|\mathcal{M}_2(\lambda)\|_{1}\leq C \displaystyle{\int_{0}^{\epsilon_\lambda}\frac{ds}{\lambda^2|\ln \lambda|^{\frac12}}}=
O\Big(\frac{\epsilon_\lambda}{\lambda^2 |\ln \lambda|^{\frac12}}\Big)
, \quad \lambda \downarrow 0.$$
\end{proof}

\begin{lemma}\label{estiS1}
Set $r>0$  and  $\lambda_0>0$. Then, for any $j\in\N$ and {$\lambda \in [0, \lambda_0)$}  
 \bel{nov11_40} 
n_{*}\left(r;S_j[\lambda, (\lambda_0, \infty) ]\right) =O(1), \qquad 
n_{*}(r;S_{j, \infty}[\lambda, (\lambda_0, \infty) ] ) =O(1).
\ee
Moreover, for {$\lambda >0$, } $\epsilon_\lambda\in (0,\lambda)$ and $I=(0,\lambda-\epsilon_\lambda)\cup(\lambda+\epsilon_\lambda,\infty)$

{\bel{nov11_4} 
n_{*}\Big(r;S_j[\lambda, I ]-S_{j,\infty}[\lambda, I ] \Big)=  O \left( \frac{| \ln \epsilon_\lambda |}{| \ln \lambda |^{3/2}} \right), \quad \lambda \downarrow 0.
\ee}
\end{lemma}
\begin{proof}
For $k\in \varrho_j(\lambda_0, \infty) = (- \infty, \varrho_j(\lambda_0))$ we have $E_j(k)-{\mathcal E}_j \geq \lambda_0$,  then the compact operator $S_{j}S_{j}^*[\lambda, (\lambda_0, \infty) ] =\widetilde{V}^{1/2}{\mathcal
F}^*\, \int_{(-\infty,\varrho_j(\lambda_0))}^\oplus
(E_j(k)-\mathcal{E}_j-\lambda)^{-1}\pi_{j}(k)\,dk\,\mathcal{F}\widetilde{V}^{1/2}$ admits a norm limit  as $\lambda$ goes to $0$, yielding  the estimate in \,\eqref{nov11_40}. The same idea works for $S_{j, \infty}[\lambda, (\lambda_0, \infty) ] $.

In order to prove \eqref{nov11_4}, using the Ky Fan inequality  \eqref{lau13} and \eqref{nov11_40}, it is sufficient to show the following estimates for any $\eta>1$ and $\lambda_0$ small enough:
\bel{ss+1}
n_{*}\Big(r;(S_j- S_{j,\infty})[\lambda, (0, \lambda - \epsilon_\lambda) ]\Big)=  O \left( \frac{| \ln \epsilon_\lambda |}{| \ln \lambda |^{3/2}} \right) ,
\ee
\bel{ss+2}
n_{*}\Big(r;(S_j- S_{j,\infty})[\lambda, (\lambda+\epsilon_\lambda,\eta \lambda ) ]\Big)=  O \left( \frac{| \ln \epsilon_\lambda |}{| \ln \lambda |^{3/2}} \right) ,
\ee
\bel{ss+3}
n_{*}\Big(r;(S_j- S_{j,\infty})[\lambda, (\eta \lambda, \lambda_0 ) ]\Big)= 0 .
\ee
First,  in the  proof of \eqref{ss+3} we use that  for $k \in \varrho_j((\eta \lambda, \lambda_0 ))$, $E_j(k)-{\mathcal E}_j \geq \eta \lambda$ and then  
$$
|E_j(k)-{\mathcal E}_j-\lambda|^{-1/2} \leq \Big( \frac{\eta}{\eta-1} \Big)^{\frac12} |E_j(k)-{\mathcal E}_j|^{-1/2},$$
which together with the min-max principle yield
\bel{19jul16}
n_*\left(r;(S_j-S_{j,\infty})[\lambda,(\eta \lambda, \lambda_0 )]\right)\leq
n_*\left(r \left(\frac{\eta-1}{\eta}\right)^{1/2};(S_j-S_{j,\infty})[0,(0, \lambda_0 )]\right).
\ee
Also, by the definition of  $S_j$ and $S_{j,\infty}$
\bel{normS}
\| (S_j-S_{j,\infty})[0, (0,\lambda_0) ] \| \leq \| \widetilde{V}\|_{L^\infty(\R^2)}^{1/2} \, 
\sup_{k > \varrho_j(\lambda_0)} \Big( |E_j(k)-{\mathcal E}_j|^{-1/2} \,
\|  \psi_j(\cdot,k)-\psi_{j,\infty}(\cdot,k)\|_{L^2(\R_+)}  \Big).
\ee
Consequently,  by Proposition \ref{mitlef} {and $\lim_{s\downarrow 0}\varrho_j(s)=\infty$},  we can find $\lambda_0>0$ sufficiently small such that the operator norm $ \| (S_j-S_{j,\infty})[0, (0,\lambda_0) ] \| < r$, and then  $n_{*}\Big(r;(S_j-S_{j,\infty})[0, (0,\lambda_0) ] \Big)=0$.

For the  proof of \eqref{ss+1} and \eqref{ss+2},  we  use \eqref{dj36} to obtain
\bel{nov11_7}\ba{l}
n_{*}\left(r;(S_j-{S}_{j,\infty})[\lambda, I]\right)\leq \frac{1}{r}\| (S_j-{S}_{j,\infty})[\lambda, I]\|_1\\[.5em]
=\displaystyle{\frac{1}{r}\int_{\R_+\times\R}V(x,y)\int_{\varrho_j(I)}
|E_j(k)-{\mathcal E}_j-\lambda|^{-1}|\psi_j(x;k)-\psi_{j,\infty}(x;k)|^{2}\,dk\,dy\,dx}\\[1em]
\leq C \displaystyle{\int_{\R_+}\int_{\varrho_j(I)}
|E_j(k)-{\mathcal E}_j-\lambda|^{-1}|\psi_{j,\infty}(x;k)-\psi_{j,\infty}(x;k)|^{2}\,dk\,dx}.\\[1em]

\ea
\ee
Thus,  regarding Proposition \ref{mitlef} and \eqref{ivri_asymp_2},  for $I= (0, \lambda - \epsilon_\lambda)$ or $I =  (\lambda+\epsilon_\lambda,\eta \lambda )$, as $\lambda \downarrow 0$, we have
$$
\ba{ll}
n_{*}\left(r;(S_j-{S}_{j,\infty})[\lambda, I]\right)\leq
 C\displaystyle{\int_{\varrho_j(I)}
|E_j(k)-{\mathcal E}_j-\lambda|^{-1}(E_j(k)-{\mathcal E}_j)\,\frac{dk}{k^2}} \\[1em]=C\displaystyle{\int_{I}\frac{s\,|\varrho'_j(s)|}{|\lambda-s|\,\varrho_j(s)^2}\,ds}
 = O \left( \displaystyle{\int_{I}\frac{1}{|\lambda-s|\, |\ln s|^{3/2}}\,ds }\right) 
= O \left( \displaystyle{\frac{| \ln \epsilon_\lambda |}{| \ln \lambda |^{3/2}} }\right).
\ea
$$
\end{proof}

\subsection{Some pseudodifferential analysis}\label{Pseudo}

For $a\in \mathcal{S}_p^q$ {(introduced in \eqref{12mar17a})}, we define the operator  $Op^W(a)$ using  the Weyl quantization:
$$(Op^{W}(a)u)(x):=\frac{1}{2\pi}\int_{\R^2}a\left(\frac{x+y}{2},\xi\right)e^{-i(x-y)\xi}u(y)\,dy\,d\xi,$$
for $u$ in the Schwartz space ${\mathcal S}(\R)$.

{In this section we will prove some lemmas that will help us  to show that, roughly, the operator {$ S_{j,\infty} S^*_{j,\infty} [\lambda, (0, \lambda-\epsilon_\lambda ) ]$ is a good approximation of ${\rm Re\,}(T_j)$, and then that} $\lambda^{-1}Op^W(V)$ is an effective Hamiltonian of our problem}.

For general properties  of pseudodifferential operators like: composition, selfadjointness, norm estimates, compact properties, etc. we refer to \cite{hor3}. Here we will consider more specifically selfadjoint pseudodifferential operators of negative order (i.e., $a \in S_p^q$ with $q>0$ and $p \geq 0$) which are known to be compact.  
{

Set $\widetilde{N}(\lambda,a)$ as the volume function in $\R^2$ 
\bel{apr20b}
\widetilde{N}(\lambda,a):=\frac{1}{2\pi}vol\{(x,\xi)\in \R^2; a(x,\xi)> \lambda\}.
\ee

\begin{lemma}\label{estipsdo1} 
Let $a \in \mathcal{S}_{p}^{q}$ be a real valued symbol, with $q>0$ {and $p\geq0$}. Then,  the operator $Op^{W}(a)$ is essentially self-adjoint,  compact and  its eigenvalue  counting functions satisfy
$$n_+(\lambda, Op^{W}(a)) + n_-(\lambda, Op^{W}(a))  \leq C 
\left \{\begin{array}{lcc} 
\lambda^{- \frac{1}{q}} & \hbox{ if } & p>q\\
\lambda^{- \frac{2}{q}} & \hbox{ if } & p=0,\\
\end{array} \right. 
$$
where $C$ depends only on a finite number of the semi-norms $$n^{p,q}_{\alpha,\beta}(a){:= \sup_{(x,\xi) \in \R^2}|\langle x\rangle^{p}\langle x,\xi\rangle ^{q+\alpha}\partial_\xi^\alpha \partial_x^\beta a(x,\xi)|,\quad \alpha,\beta \in\Z_+.}$$
\end{lemma}
\begin{proof}
Applying \cite[Lemma 4.7]{dauro} we have 
$$n_+(\lambda, Op^{W}(a)) + n_-(\lambda, Op^{W}(a))  \leq C \, 
\widetilde{N}(\lambda,\langle x\rangle ^{-p}\langle x,\xi\rangle ^{-q}), $$
where $C$ depends only on a finite number of semi-norms $n^{p,q}_{\alpha,\beta}(a)$.
Using that $\Omega^q_p(\lambda):= \{ (x,\xi) \in \R^2; \; \langle x\rangle ^{-p}\langle x,\xi\rangle ^{-q} \geq \lambda \}$ satisfies 
$$\Omega^q_p(\lambda) = \{ (x,\xi) \in \R^2; \; \langle x\rangle ^{2} + \xi^2 \leq \lambda^{- \frac{2}{q}} \, \langle x\rangle ^{- \frac{2p}{q}} \}
\subset \{ (x,\xi) ; \; \langle x\rangle ^{1 + \frac{p}{q}} \leq \lambda^{- \frac{1}{q}}; \;   |\xi | \leq \lambda^{- \frac{1}{q}}  \, \langle x\rangle ^{- \frac{p}{q}} \},$$
we obtain
$$\widetilde{N}(\lambda,\langle x\rangle ^{-p}\langle x,\xi\rangle ^{-q}) \leq 2 \lambda^{- \frac{1}{q}} \, \int_{ |x | \leq \lambda^{- 1/(p+q)} }  \langle x\rangle ^{- \frac{p}{q}}\,dx,$$
and the claimed estimate follows.
\end{proof}

 Let $W\geq0$ {be} a function in $L^1(\R^2$). Set $Q_j(W):L^2(\R)\to L^2(\R^2)$ to be  the integral operator with kernel
\bel{nov5_2_0}
(2\pi)^{-1/2}{W}(x,y)^{1/2} e^{iky}\psi_{j,\infty}(x;k),
\ee
and define  the operator  $ \mathcal{W}_j:L^2(\R)\to L^2(\R) $ by
\bel{S0Vj}
\mathcal{W}_j={Q_j}^*(W) \, Q_j(W).
\ee
{Note that  for $\lambda \notin I$, the operators $ {S}_{j,\infty}[\lambda, I]$ and  $Q_j(\widetilde{V} )$ are connected through  the expression}
\bel{SS0}
{S}_{j,\infty}[\lambda, I] =  Q_j(\widetilde{V}) \,  \, {\bf 1}_{\varrho_j(I)}(\cdot) \,  \, |E_j(\cdot)-{\mathcal E}_j-\lambda|^{-1/2}.
\ee
{Furthermore,  we have the following proposition}
\begin{proposition}[\cite{shi2}, Lemma 5.1]\label{shirai} Suppose that $W$ is in $\mathcal{S}_0^m$. Then,  for any $j\in\N$ we have $${\mathcal W}_j=Op^W(w_j),$$ 
where the symbol $w_j$ is in ${\mathcal S}_0^m$. Further
\bel{sep25}
w_j(x,\xi)={W}(x,-\xi)+R_1(x,\xi),
\ee
where  $R_1\in {\mathcal S}^{m+1}_0$  \end{proposition}

\begin{lemma}\label{estiS2_a}
Let  $W\geq0$   be a function in $\mathcal{S}_0^m$  and 
$\rm{supp}\, W\subset [K_0,\infty)\times\R$, for some $K_0\in \R$.  Then
$$ 
n_{+}\Big(\lambda ; {\bf 1}_{(-\infty, K_0-1)}(\cdot) \, \mathcal{W}_j \, {\bf 1}_{(-\infty,K_0-1)}(\cdot) \Big) = O(\lambda^{-\frac{1}{m}}),\quad \lambda\downarrow 0.
$$
\end{lemma}
\begin{proof}
Consider a smooth function $\chi$ such that $ 0\leq \chi \leq 1$, $\chi(k)=1$ for  $k \in (-\infty,K_0-1]$ and $\chi(k)=0$ for {$k \in [K_0-\frac12,\infty)$}. 
We clearly have
\bel{tur0}
n_{+}\Big(\lambda ; {\bf 1}_{(-\infty, K_0-1)}(\cdot) \, \mathcal{W}_j \, {\bf 1}_{(-\infty, K_0-1)}(\cdot) \Big)
\leq
n_{+}\Big(\lambda ; \chi(\cdot)\,  \mathcal{W}_j \, \chi(\cdot) \Big).
\ee
{To finish the proof of the Lemma,  it}  is sufficient to show that $\chi \, {\mathcal W}_j \, \chi$ is a pseudo-differential operator in $Op^W(\mathcal{S}^m_p)$ for any $p \geq 0$, and then  apply Lemma \ref{estipsdo1}.  Moreover, since any derivative of  $\chi$ is still supported in $(- \infty, K_0)$, using composition theorems, we are reduced to prove that for any $p\geq 0$ and any $(\alpha, \beta) \in \N^2$ there exists a positive constant $C(m,p, \alpha, \beta)$ such that 
\bel{june24}
{\forall (x,\xi) \in (- \infty, K_0-\frac12)\times \R }, \qquad 
|\langle x \rangle^{p}\langle x,\xi \rangle^{m+\alpha}\partial_x^\beta\partial_\xi^\alpha w_j(x,\xi)| 
\leq C(m,p, \alpha, \beta), 
\ee
where $w_j(x,\xi)$ is the Weyl's symbol of ${\mathcal W}_j $ given in Proposition \ref{shirai}. 

Suppose first that $W$ is in the Schwartz space $S(\R^2)$, then $w_j(x,\xi)$ is given by
$$
w_j(x,\xi)=\frac{1}{2\pi}\int_{\R^3}e^{-izy'}W(x',z-\xi)\psi_{j,\infty}(x';x-\frac{y'}{2})\psi_{j,\infty}(x';x+\frac{y'}{2})\,dx'\,dy'\,dz,
$$
{and satisfies}
$$\ba{ll}  |\langle x,\xi \rangle^{m+\alpha}\partial_x^\beta\partial_\xi^\alpha w_j(x,\xi)|
\leq
C \langle x,\xi \rangle^{m+\alpha}\\[,5em]
\times  \displaystyle{\int_{\R^3} \left|\partial^\alpha_\xi W(x',z-\xi)\right|}
\langle z\rangle^{-2N} 
\displaystyle{\left| \partial_x^\beta\langle D_{y'}\rangle^{2N} \left( \psi_{j,\infty}(x';x-\frac{y'}{2})\psi_{j,\infty}(x';x+\frac{y'}{2})\right)\right|
\,dx'\,dy'\,dz}\\[1em]
\leq  C\displaystyle{\int_{\R^3}\langle x',z-\xi \rangle^{m+\alpha} \left| \partial^\alpha_\xi W(x',z-\xi)\right| \langle z\rangle^{-2N+m+\alpha}}\\[1em]
\displaystyle{\times\left|\langle x-bx' \rangle^{m+\alpha}
 \partial_x^\beta\langle D_{y'}\rangle^{2N} \left( \psi_{j,\infty}(x';x-\frac{y'}{2})\psi_{j,\infty}(x';x+\frac{y'}{2})\right)\right|
\,dx'\,dy'\,dz,}
\ea
$$
where in the last inequality we have used that $\langle x,\xi \rangle \leq C \langle z \rangle \langle x-bx' \rangle \langle x',z-\xi\rangle$, for some positive constant $C$.
Choosing $N$ sufficiently large and using  that $W$ is supported on $[K_0,+\infty)\times \R$, we deduce
$$\ba{ll}|\langle x,\xi \rangle^{m+\alpha}\partial_x^\beta\partial_\xi^\alpha w_j(x,\xi)|\leq{C\, n_{\alpha,0}^{0,m}(W)}\\ 
\times\displaystyle{\int_{[K_0,+\infty)\times \R} \langle x-bx' \rangle^{m+\alpha}}{\left|\partial_x^\beta\langle D_{y'}\rangle^{2N} \left( \psi_{j,\infty}(x',x-\frac{y'}{2})\psi_{j,\infty}(x',x+\frac{y'}{2})\right)\right|
\,dx'\,dy'.}
\ea
$$
 Furthermore, straightforward calculations show that:
$$\partial_x^{\beta} \partial_{y'}^{N} \psi_{j,\infty}(x',x-\frac{y'}{2})\psi_{j,\infty}(x',x+\frac{y'}{2})=P(b^{1/2}x'-b^{-1/2}x,y') e^{-(b^{1/2}x'-b^{-1/2}x)^2-\frac{1}{4b}y'^2},$$
where $P$ is a polynomial. Thus, we conclude  that there exists $M>0$ and $C_{m,p,\alpha,\beta}$ (depending on $n_{\alpha,0}^{0,m}(W)$) such that
\bel{oct13}
|\langle x,\xi \rangle^{m+\alpha}\partial_x^\beta\partial_\xi^\alpha w_j(x,\xi)|
\leq C_{m,p,\alpha,\beta}\, \displaystyle{\int_{[K_0,+\infty)} \langle x-bx' \rangle^{M}}    e^{-(bx'-x)^2b^{-1}} dx'.\\
\ee
The r.h.s. of \eqref{oct13} is clearly exponentially decreasing with respect to $x<K_0-1$, implying  \eqref{june24} for  $W$ in Schwartz space.  Using a   limiting argument we {can} deduce \eqref{june24}
for $W \in S^{m}_0$.
\end{proof}

{\begin{lemma}\label{estiS2} Let  $W\geq0$  be  in $\mathcal{S}_0^m$ and 
$\rm{supp}\, W\subset [K_0,\infty)\times\R$, for some $K_0\in \R$. {Further},  let  $\epsilon_\lambda \in (\lambda^\theta, \delta \lambda)$,  $\theta>1$,  and $ \delta  \in (0,1)$.
Then, for any $A>0$ and $\nu>0$
\bel{may31_1} 
n_{+}\Big(\epsilon_\lambda ; {\bf 1}_{(-\infty, \varrho_j( A \lambda))}(\cdot) \, \mathcal{W}_j \, {\bf 1}_{(-\infty,\varrho_j(A\lambda))}(\cdot) \Big) = o({\epsilon_\lambda}^{-\frac{1}{m}}\, \lambda^{-\nu}),\quad \lambda\downarrow 0.
\ee
\end{lemma}
\begin{proof}
Thanks to \eqref{S0Vj} and the 
Ky Fan inequality
\bel{30jun16a}\ba{rl}
n_{+}\Big(\epsilon_\lambda ; {\bf 1}_{(-\infty, \varrho_j(A\lambda))}(\cdot)
\mathcal{W}_j \, {\bf 1}_{(-\infty, \varrho_j(A\lambda))}(\cdot) \Big)
\\
\leq & n_{+}\Big(\frac{\epsilon_\lambda}{4} ; {\bf 1}_{(-\infty, K_0-1)}(\cdot)\mathcal{W}_j \, {\bf 1}_{(-\infty,K_0-1)}(\cdot)\Big)\\[.7em]
+&n_{+}\Big(\frac{\epsilon_\lambda}{4} ; {\bf 1}_{(K_0-1,  \varrho_j(A\lambda))}(\cdot)\mathcal{W}_j \, {\bf 1}_{(K_0-1, \varrho_j(A\lambda))}(\cdot)\Big).
\ea
\ee
{The desired estimate} for the first term in the r.h.s. of  \eqref{30jun16a} follows from Lemma \ref{estiS2_a} {with $\epsilon_\lambda$ instead of $\lambda$}. In order to estimate the second term, we will need the following inequality (in the sense of quadratic forms). When $\varrho_j(I)$ is a bounded set, for any $p \geq 0$
\bel{ubS0}
Q_j(W)\, {\bf 1}_{\varrho_j(I)}(k) \, Q_j^*(W) \leq \Big( \sup_{k \in \varrho_j(I)}  \langle k \rangle^{p}   \Big)  \, 
Q_j(W) \, \langle \cdot \rangle^{-p} \,  Q_j(W)^*.
\ee
Consequently,  using \eqref{S0Vj}, that the postive eigenvalues of $TT^*$ coincide with those of $T^*T$ ({here we use it} for $T= Q_j(W)  {\bf 1}_{(K_0-1,  \varrho_j(A\lambda))}(\cdot) $ and for $T=   Q_j(W)  \langle \cdot \rangle^{-p/2}$), and    the min-max pinciple, we conclude that
\bel{9aug16}
\ba{ll}n_{+}\Big(\epsilon_\lambda ; {\bf 1}_{(K_0-1,  \varrho_j(A\lambda))}(\cdot)\mathcal{W}_j \, {\bf 1}_{(K_0-1, \varrho_j(A\lambda))}(\cdot)\Big)&=n_{+}\Big(\epsilon_\lambda ; Q_j(W)  {\bf 1}_{(K_0-1,  \varrho_j(A\lambda))}(\cdot)      Q_j(W)^*\Big)\\[.7em]
&\leq n_{+}\Big(\epsilon_\lambda c_\lambda ; Q_j(W)   \langle \cdot \rangle^{-p}    Q_j(W)^*\Big)\\[.7em]
&=n_+\Big(\epsilon_\lambda c_\lambda, \langle \cdot \rangle^{-p/2} \mathcal{W}_j \langle \cdot \rangle^{-p/2} \Big) , \ea
\ee 
where $c_\lambda=1/ \sup \{ \langle k \rangle^{p/2}; k\in (K_0-1, \varrho_j(A\lambda))    \}$. 
 Since $\langle k \rangle^{-p} \in \mathcal{S}_p^0$,  Proposition \ref{shirai}  together with  the composition formula give us that 
$\langle \cdot \rangle^{-p/2} \mathcal{W}_j \langle \cdot \rangle^{-p/2} =Op^W(a)$, where the  symbol $a\in\mathcal{S}_p^m$, and $p\geq0$. Using  Lemma \ref{estipsdo1}   with $p $ sufficiently large,  we obtain that $n_+(\lambda, \langle \cdot \rangle^{-p/2} \mathcal{W}_j \langle \cdot \rangle^{-p/2} ) = O ({\lambda}^{- \frac{1}{m}}),$ as  $ \lambda  \downarrow 0$.  {Combining this estimate with}  \eqref{9aug16} we  deduce 
\bel{30jun16c}
{ n_{+}\Big({\epsilon_\lambda} ; {\bf 1}_{(K_0-1,  \varrho_j(A\lambda))}(\cdot)\mathcal{W}_j \, {\bf 1}_{(K_0-1, \varrho_j(A\lambda))}(\cdot)\Big)
=O({(\epsilon_\lambda c_\lambda)}^{-\frac{1}{m}})}.
\ee
Now, regarding   \eqref{ivri_asymp_2},  ${c_\lambda}^{-\frac{1}{m}}= o(\lambda^{-\nu})$ for any  $\nu>0$, {and then, \eqref{30jun16c} implies  \eqref{may31_1}.}

\end{proof}

\begin{proposition}\label{asympS}
Let  $W\geq0$ be  in $\mathcal{S}_0^m$ and 
$\rm{supp}\, W\subset [K_0,\infty)\times\R$ for some $K_0\in \R$. 
Let $r>0$ and $r_\pm >0$ such that $r_- < r<r_+$.  Then for any $\nu>0$
\bel{may24_4} \ba{ll}&
n_+(r_+ \lambda,  \mathcal{W}_j) + {o(\lambda^{-\frac{1}{m}}\lambda^{-\nu})}\\ [.3em] 
\leq& n_{*}\Big(r^{\frac12};Q_j(W) \,  \, {\bf 1}_{(\varrho_j( \lambda-\epsilon_\lambda),\infty)}(\cdot) \,  \, |E_j(\cdot)-{\mathcal E}_j-\lambda|^{-1/2}\Big)\\[.3em]
\leq &
n_+(r_- \lambda,  \mathcal{W}_j) + { o(\epsilon_\lambda^{-\frac{1}{m}}\lambda^{-\nu}),} \quad \lambda \downarrow 0,
\ea\ee
for $\epsilon_\lambda \in (\lambda^\theta, \delta \lambda)$, $\theta>1$ and $ \delta  \in (0,1)$.
\end{proposition}
\begin{proof}
For the upper bound, consider the inequalities  $|E_j(k)-{\mathcal E}_j-\lambda| \geq (1-\delta) \lambda,$ valid for $k \in (\varrho_j(\delta \lambda ), \infty)$, and $|E_j(k)-{\mathcal E}_j-\lambda| \geq \epsilon_\lambda,$ valid for $k \in (\varrho_j(\lambda-\epsilon_\lambda ), \varrho_j(\delta \lambda ))$. Then,  for  $r_1,r_2>0$ the Ky Fan inequality and the min-max principle imply
\bel{june1_4}\ba{ll}
&n_{*}\Big(r_1+r_2;Q_j(W) \,  \, {\bf 1}_{(\varrho_j( \lambda-\epsilon_\lambda),\infty)}(\cdot) \,  \, |E_j(\cdot)-{\mathcal E}_j-\lambda|^{-1/2}\Big) \\[.5em]
\leq & n_{*}\Big(r_1 \epsilon_\lambda^{1/2};Q_j(W) \,  \, {\bf 1}_{(\varrho_j(\lambda-\epsilon_\lambda ), \varrho_j( \delta\lambda) )}(\cdot)\Big)\\[.5em]
+& n_{*}\Big(r_2(1-\delta)^{1/2}\lambda^{1/2} ;Q_j(W) \,  \, {\bf 1}_{(\varrho_j(\delta\lambda),\infty)}(\cdot) \Big).
\ea
\ee
{The min-max principle together with \eqref{S0Vj} and Lemma \ref{estiS2} imply
\bel{10aug16}n_{*}\Big( \epsilon_\lambda^{1/2};Q_j(W) \,  \, {\bf 1}_{(\varrho_j(\lambda-\epsilon_\lambda ), \varrho_j( \delta\lambda) )}(\cdot)\Big)\leq
n_{*}\Big(\epsilon_\lambda^{1/2};Q_j(W) \,  \, {\bf 1}_{(-\infty, \varrho_j( \delta \lambda) )}(\cdot)\Big) =o( \epsilon_\lambda^{-\frac{1}{m}}\lambda^{-\nu}),
\ee
for any $\nu>0$, and 
\bel{13jul16}
n_{*}\Big(r_2 (1-\delta)^{1/2}\lambda^{1/2} ;Q_j(W) \,  \, {\bf 1}_{(\varrho_j(\delta\lambda),\infty)}(\cdot) \Big)
\leq
n_+(r_2^2 (1-\delta) \lambda ;  \mathcal{W}_j).
\ee
Inequalities \eqref{june1_4}-\eqref{13jul16}  give us the upper bound.}

Concerning the lower bound, we use that  for $k \in (\varrho_j(\lambda - \epsilon_\lambda), \infty )$ the band function satisfies $|E_j(k)-{\mathcal E}_j-\lambda| =   \lambda- (E_j(k)-{\mathcal E}_j) < \lambda,$ 
then 
\bel{june1_2}
n_{*}\big(r^{1/2};Q_j(W) \,  \, {\bf 1}_{(\varrho_j( \lambda-\epsilon_\lambda),\infty)}(\cdot) \,  \, |E_j(\cdot)-{\mathcal E}_j-\lambda|^{-1/2} \big) \geq 
n_*\big( {r^{1/2}{\lambda}^{1/2}};  Q_j(W) \,  \, {\bf 1}_{(\varrho_j( \lambda-\epsilon_\lambda),\infty)}(\cdot)  \big).
\ee

{Gathering the Weyl inequalities \eqref{lau11} with \eqref{S0Vj}, we have
$$
\ba{ll}
n_*\Big( {r^{1/2}{\lambda}^{1/2}};  Q_j(W) \,  \, {\bf 1}_{(\varrho_j( \lambda-\epsilon_\lambda),\infty)}(\cdot)  \Big)  &\geq  n_+(r_+ \lambda ;  \mathcal{W}_j) \\ 
   &- n_*\Big( {(r_+-r)^{1/2}{\lambda}^{1/2}};  Q_j(W) \,  \, {\bf 1}_{(- \infty, \varrho_j( \lambda-\epsilon_\lambda))}(\cdot)  \Big).
 \ea 
 $$
Then, since $\varrho_j( \lambda-\epsilon_\lambda) \leq \varrho_j( \lambda(1-\delta))$, {Lemma}  \ref{estiS2}  implies the lower bound.}

\end{proof}

\subsection{Proof of Theorem \ref{V_power}}\label{proof_V_power}

{Let us recall  that,} for $\lambda >0$,
$\epsilon_\lambda \in (\lambda^\theta, \delta \lambda)$, $\theta>1$ and  $ \delta  \in (0,1)$, {${\rm Re}\,T_j({\mathcal E}_j+\lambda)$ is given by (see \eqref{decomp+}):}
$$\mbox{p.v.}\displaystyle{\int_{\lambda -\epsilon_\lambda }^{\lambda+ \epsilon_\lambda}
\frac{G_{j}(\varrho_j(s))^*G_{j}(\varrho_j(s))}{{\lambda-s}}\varrho_j'(s)ds}
+ S_{j} S_{j} ^* [\lambda, (\lambda+ \epsilon_\lambda ,  \infty) ] - S_{j} S_{j} ^* [\lambda, (0, \lambda-\epsilon_\lambda ) ].
$$
Accordingly,  the Weyl inequalities \eqref{lau11},  Lemma \ref{estiPV} and Lemma \ref{estiS1}, imply that  for any $\rho\in (0,1)$
\bel{11_jul_16}\ba{ll}&
n_+ (r(1+\rho) , S_{j, \infty} S_{j, \infty} ^* [\lambda, (0, \lambda- \epsilon_\lambda ) ] )
+O\left(\frac{\epsilon_\lambda}{\lambda^2 |\ln \lambda|^{\frac12}}\right)\\
&{-  n_+} (r\rho/2 , S_{j, \infty} S_{j, \infty} ^* [\lambda, (\lambda+ \epsilon_\lambda ,  \infty) ] ) \\\leq&
n_-( r, {\rm Re}\,T_j({\mathcal E}_j+\lambda) ) \\\leq &
n_+ (r(1-\rho) , S_{j, \infty} S_{j, \infty} ^* [\lambda, (0, \lambda- \epsilon_\lambda ) ] )+O\left(\frac{\epsilon_\lambda}{\lambda^2 |\ln \lambda|^{\frac12}}\right)
 .
\ea
\ee

{Using} \eqref{SS0},  exploiting that for $k \in \varrho_j(\lambda+ \epsilon_\lambda ,  \infty)$, $|E_j(k)-{\mathcal E}_j-\lambda | \geq \epsilon_\lambda$, and  following  the proof of Lemma \ref{estiS2}, we obtain
\bel{12sept}\ba{ll}
n_+ (r , S_{j, \infty} S_{j, \infty} ^* [\lambda, (\lambda+ \epsilon_\lambda ,  \infty) ] ) 
= o({\epsilon_\lambda}^{-\frac{1}{m}}\, \lambda^{-\nu}),\quad \lambda\downarrow 0,\ea\ee
for any $\nu>0$. {This together with Theorem \ref{firsteh}, \eqref{11_jul_16},  and choosing $\epsilon_\lambda =\lambda^{\theta}$, with $\theta=2-2/m$, imply
\bel{dec23z}\ba{ll}&
n_+ (r(1+\rho) , S_{j, \infty} S_{j, \infty} ^* [\lambda, (0, \lambda- \epsilon_\lambda ) ] )+o(\lambda^{-2/m})\\
\leq & \pm \xi(\mathcal{E}_j\pm\lambda;H_+,H_0) \\ \leq & n_+ (r(1-\rho) , S_{j, \infty} S_{j, \infty} ^* [\lambda, (0, \lambda- \epsilon_\lambda ) ] )+o(\lambda^{-2/m}),\ea\ee
as $\lambda \downarrow 0.$}

Now, consider two smooth functions $V^\pm$, defined on $\R^2$, such that they  satisfy \eqref{12apr17},  $ 0\leq V^- \leq  \widetilde{V} \leq V^+$ and the differences  $\widetilde{V}- V^\pm$ are compactly supported in the $x$-direction. Then,  for the volume functions defined in   \eqref{apr20} and \eqref{apr20b} it is easy to prove that  
\bel{30jun16}\widetilde{N}(\lambda, V^\pm)= \widetilde{N}(\lambda,  \widetilde{V}) +O (\lambda^{-\frac{1}{m}}) = N(\lambda,  V) + O (\lambda^{-\frac{1}{m}}). \ee
In particular, $\widetilde{N}(\lambda, V^\pm)$ satisfy {\eqref{12apr17}\rm{b}}  and {\eqref{12apr17}\rm{c}.}  Therefore,  for any $\rho\in(0,1)$
\bel{29jul16a}
\lim_{\rho\downarrow 0}\frac{\widetilde{N}(\lambda(1\pm\rho),V^\pm)}{\widetilde{N}(\lambda,V^\pm)}=1.\ee}
{Further, an easy computation  shows that}  
\bel{S0Vj2}
\mathcal{V}_j^-  \leq   \mathcal{V}_j  \leq   \mathcal{V}_j^+ ,\ee
{where $\mathcal{V}_j^\pm:= {Q_j}^*(V^\pm) \, Q_j(V^\pm)$, $\mathcal{V}_j:= {Q_j}^*(\widetilde{V}) \, Q_j(\widetilde{V})$ 
are  operators defined 
as in \eqref{S0Vj}. }

In consequence, taking into account \eqref{dec23z},  \eqref{SS0},  {Proposition} \ref{asympS}, \eqref{S0Vj}  and \eqref{S0Vj2},  
we obtain that  for all $\rho \in (0,1)$, 
 \bel{dec23}
 n_+(\lambda (1+\rho) ;\mathcal{V}_j^-)+o(\lambda^{-2/m})
\leq \pm \xi(\mathcal{E}_j\pm\lambda;H_+,H_0) \leq  n_+(\lambda (1-\rho) ;\mathcal{V}_j^+)+o(\lambda^{-2/m}),\ee
as $\lambda \downarrow 0.$
To finish he proof of Theorem \ref{V_power} we need the following result.

 \begin{proposition}\label{asympsdo1} \cite[Theorem 1.3]{dauro}
Let $a \in \mathcal{S}_0^m$ be a real valued symbol, with $m>0$.
Assume that the volume functions $\widetilde{N}(\lambda, \pm a)$ defined by \eqref{apr20b} satisfy {\eqref{12apr17}\rm{b},}   {\eqref{12apr17}\rm{c}}. 
Then, there exists $\nu>0$ such that, as $\lambda  \downarrow 0$,  the counting functions satisfy
$$n_\pm(\lambda, Op^{W}(a))   =  \widetilde{N}(\lambda, \pm a) (1 + O( \lambda^{\nu})).$$
\end{proposition}

Hence, putting together \eqref{dec23}, Propositions \ref{shirai}, \ref{asympsdo1}, the Weyl inequalities, Lemma \ref{estipsdo1}, \eqref{30jun16} and \eqref{29jul16a},  we obtain \eqref{the4}.

\section*{ Acknowledgments}  The  authors thank Georgi  Raikov for suggesting this problem and useful discussions.

\vspace{1pt}

\noindent V. Bruneau was partially supported by ANR 2011 BS01019 01.

\noindent P. Miranda was partially supported by CONICYT FONDECYT Iniciaci\'on 11150865.

\bibliographystyle{plain}
\bibliography{biblio, bilbio}

\end{document}